\documentclass[11pt]{amsart}

\usepackage{eucal,dsfont,txfonts}

\usepackage[text={6.5in,9in},centering]{geometry} 

\newcommand{\inv}{^{\raisebox{.2ex}{$\scriptscriptstyle-1$}}}

\newcommand{\la}{Q}
\newcommand{\id}{\mathcal{I}_\la}
\newcommand{\rd}{\mathcal{R}}
\newcommand{\idp}{\mathcal{I}^+_\la}
\newcommand{\nl}{\mathcal{N}}
\newcommand{\an}{\mathcal{A}_{\la}}

\newcommand{\spc}{\mathrm{Spec}_\la}

\newcommand{\pr}{\mbox{{\small{$\&$}}}}

\usepackage[dvipsnames]{xcolor}   
\usepackage{xparse}
\usepackage{xr-hyper}
\usepackage[linktocpage=true,colorlinks=true,hyperindex,citecolor=teal,linkcolor=blue]{hyperref}

\newtheorem{theorem}{Theorem}[section]
\newtheorem{proposition}[theorem]{Proposition}
\newtheorem{lemma}[theorem]{Lemma}
\newtheorem{corollary}[theorem]{Corollary}
\theoremstyle{definition}
\newtheorem{definition}{Definition}[section]
\newtheorem{remark}[theorem]{Remark}
\newtheorem{example}[theorem]{Example}
\numberwithin{equation}{section}
\newtheorem{cremark}[theorem]{Concluding Remarks}
\numberwithin{equation}{section}

\begin{document}
	
	\author{Amartya Goswami}
	{\scriptsize{\address{[1] Department of Mathematics and Applied Mathematics, University of Johannesburg, South Africa;
				[2] National Institute for Theoretical and Computational Sciences (NITheCS), South Africa.}}}
	{\scriptsize{\email{agoswami@uj.ac.za}}}
	
	\title{On ideals in quantales, I}
	
	\date{}
	
	
	
	
	
	\maketitle

	\begin{abstract}
		Taking a ring-theoretic perspective as our motivation, the main aim of this series is to establish a comprehensive theory of ideals in commutative quantales with an identity element. This particular article focuses on an examination of several key properties related to ideals in quantale, including prime, semiprime, radical, primary, irreducible, and strongly irreducible ideals. Furthermore, we investigate the primary decomposition problem for quantale ideals. In conclusion, we present a set of future directions for further exploration, serving as a natural continuation of this article.
	\end{abstract}
	
	\section{Introduction}

	As pointed out in \cite{KP08}, a quantale can be regarded as a special case within various algebraic and categorical perspectives. Here are a few examples:
	
	\begin{enumerate}
		\item[$\bullet$]  complete multiplicative lattices,
		
		\item[$\bullet$] complete residuated lattices,
		
		\item[$\bullet$] semirings with infinitary sums,
		
		\item[$\bullet$] thin closed monoidal categories,
		
		\item[$\bullet$] monoids in complete semilattices.
	\end{enumerate}
	
	For the purpose of developing an ideal theory of quantales, we will adopt the first approach mentioned above.
	Ideal theory  in an algebraic structure provides a powerful framework for understanding the structural properties and behaviour of ideals within the algebraic system. It focuses on the investigation of properties and characteristics of ideals in algebraic structures, such as their generation, factorization, and intersection. It explores the interplay between ideals and other fundamental algebraic concepts, such as modules, prime ideals, and quotient structures. One of the primary goals of ideal theory is to establish connections between ideals and the overall structure of the algebraic system in which they reside.

	Although abstract ideal theory has been extensively studied (see, e.g., \cite{And74, AJ84, AJA95,  Dil62, Dil36, War37, WD39, Gos24}) for various classes of lattices by introducing different types of elements (replacing various types of ideals), there is a notable absence of detailed study on semiprime, primary, strongly irreducible, \textit{etc}., types of elements in the context of multiplicative lattices\footnote{ however, definitions of semiprime and primary ideals have been mentioned in \cite[Definition 2.6]{YX13} (see also \cite[Definition 2.2]{LW14}) in order to in order to introduce rough semiprime and rough primary ideals in a quantale. Also, a definition of prime ideal and completely prime ideal has been introduced in \cite[Definition 3 and Definition 4]{BAS16} to show when they coincide.}. It is worth noting that in a lattice without a multiplication operation, the notions of prime and strongly irreducible elements coincide. However, in the context of rings, these two types of ideals are distinct from each other.
	
	The primary objective of this series of papers is to explore the extent to which the ideal theory of rings can be extended to quantales. In many instances, the reader  finds that the proofs of various properties for different classes of ideals in rings can be adapted almost identically for ideals in quantales. However, we shall present complete proofs for all results unless they are trivial. This approach serves two main purposes:
	Firstly, providing comprehensive proofs  contributes to the understanding and advancement of this field of study. By presenting detailed arguments, we aim to enhance the readers' comprehension and facilitate further developments in this area.
	Secondly, we intend to showcase the reliance on multiplicative ideal theory principles by explicitly demonstrating the application of these principles in quantale analysis. Through the presentation of proofs, we aim to emphasize the significance of ideal-theoretic aspects in the study of quantales.
	
	Here is an explanation of the terminology and notation used in this paper.
	The symbol $\mathds{N}$ represents the set of natural numbers, \textit{i.e.}, ${1, 2, 3, \ldots}$. When we have a set $X$ and $S$ is a subset of $X$, we denote the complement of $S$ with respect to $X$ as $X\neg S=\{x\in X\mid x\notin S\}$. The empty set is denoted by $\emptyset$, while the top and bottom elements of a lattice $(\mathcal{L}, \preccurlyeq, \vee, \wedge)$ are represented as $\top$ and $\bot$, respectively.
	Throughout the paper, the default ordering relation for the poset of subsets of a set  is the set inclusion relation, denoted by $\preccurlyeq$.
	
	
	
	The paper is organized as follows. In \textsection \ref{kid}, we provide an introduction to the concept of ideals in quantales and explore the construction of new ideals using operations between them. We examine the general properties of ideals and investigate their behavior under quantale homomorphisms. Additionally, we briefly discuss some key results pertaining to maximal ideals in quantales.
	In \textsection \ref{psi}, we study prime and semiprime ideals in quantales. We examine the notion of radicals of ideals and establish an equivalence between semiprime ideals and radical ideals in a quantale. Throughout this analysis, the significance of multiplicatively closed subsets becomes apparent as they play a crucial role in the development of the theory.
	Moving forward,  \textsection \ref{ppd} focuses on primary ideals in quantales and primary decompositions. We explore the properties of primary ideals and investigate their decompositions. This section sheds light on the fundamental aspects of primary ideals within the context of quantales.
	Lastly, in  \textsection\ref{isii}, we delve into the discussion of irreducible and strongly irreducible ideals in quantales. We present a representation theorem, stating that any ideal in a Noetherian quantale can be expressed as an intersection of a finite number of irreducible ideals. Furthermore, we partially characterize arithmetic quantales by examining their strongly irreducible ideals.
	
	\section{Ideals}

	\subsection{Operations on ideals}\label{kid} 
	
	Quantales can be seen as a generalization of rings, where the underlying additive abelian groups are replaced by sup-lattices and the multiplication part as a semigroup. Since in our investigation, we aim to extend the ideal theory associated with commutative rings with identity, our focus shall be on a specific type of quantales that align with this goal. For a more comprehensive exploration of the algebraic aspects of quantales in general, we recommend  \cite{Ros90}.

	\begin{definition}\label{mld}
		An \emph{integral commutative quantale} is a complete lattice $(\la, \preccurlyeq, \bot, \top)$ endowed with an operation $\pr $, satisfying the following axioms:
		\begin{enumerate}
			
			\item\label{mla} $x\pr  (y\pr  z)=(x\pr  y)\pr  z$,
			
			\item\label{mlc} $x\pr  y = y\pr  x$,
			
			\item\label{jdp} $x\pr  (\bigvee_{\lambda \in \Lambda} y_{\lambda})=\bigvee_{\lambda \in \Lambda}(x\pr  y_{\lambda})$, 
			
			\item \label{mid} $x\pr  \top=x$,
		\end{enumerate}
		for all $x,$ $y,$ $y_{\lambda}$, $z\in \la$, and for all $\lambda \in \Lambda$, where $\Lambda$ is an index set.
	\end{definition}
	
	Throughout this paper, unless otherwise stated, by a ``quantale,'' we shall always refer to a commutative quantale with the top element  $\top$ as the identity with respect to $\pr$. We shall also use the notation $x^n$ to denote $x\pr  \cdots \pr  x$ (repeated $n$ times).

	\begin{remark}
		The aforementioned Definition \ref{mld} was first introduced in \cite{WD39} (also see \cite{Dil62}), where the authors referred to it as a \emph{multiplicative lattice}.
	\end{remark}
	
	
	
	In the following lemma, we compile a number of elementary results on quantales that will  be  utilized in the sequel.

	\begin{lemma}\label{bip}
		Let $\la$ be a quantale. Then the following holds.
		\begin{enumerate}
			\item \label{mul}
			$x\pr y\preccurlyeq x\wedge y$, for all $x, y\in \la$.
			
			\item $x\pr \bot=\bot$, for all $x\in \la$.
			
			\item\label{mon} If $x\preccurlyeq y$, then $x\pr z\preccurlyeq y\pr  z,$ for all $z\in \la$.
			
			\item\label{monj} If $x\preccurlyeq y$ and $u\preccurlyeq v$, then $x\pr u\preccurlyeq y\pr v$, for all $x,$ $y,$ $u,$ $v\in \la$.
			
			\item\label{binom} If $n\in \mathds{N}$ and if $x,$ $y\in \la$, then
			\begin{equation}\label{bif}
				(x\vee y)^{n}=\bigvee_{k=0}^{n} \binom{n}{k}a^{n-k}\pr b^k,\qquad\text{where,}\;\;\; \binom{n}{k}=\dfrac{n!}{k!(n-k)!}.
			\end{equation} 
		\end{enumerate}
	\end{lemma} 
	
	\begin{proof}
		(1) Notice that $x=x\pr \top=x\pr (y\vee \top)=(x\pr y)\vee x,$ and this implies $x\pr y\preccurlyeq x$. Similarly, we have $x\pr y\preccurlyeq y,$ and hence $x\pr y\preccurlyeq x\wedge y.$

		(2) Applying (1), we obtain $x\pr \bot\preccurlyeq x\wedge \bot =\bot,$ and hence $x\pr \bot=\bot.$
		
		(3) Since $x\preccurlyeq y$, we have $y\pr z=(x\vee y)\pr z=(x\pr z)\vee (y\pr z)$, which implies that $x\pr z\preccurlyeq y\pr z.$
		
		(4) Since $x\preccurlyeq y$ and $u\preccurlyeq v$, by (3), we have $x\pr u\preccurlyeq y\pr u$ and $y\pr u\preccurlyeq y\pr v$. From these, we obtain the claim.
		
		(5) For $n=1$, the formula (\ref{bif}) is trivially true. Suppose the formula is true for $n=k$. Using the conditions of Definition \ref{mld}, it is a routine exercise to check that the formula is also true for $n=k+1$. Hence, by induction, we have the desired claim.
	\end{proof} 
	
	We shall now present the definition of the central concept in this paper, which pertains to the notion of ideals in a quantale.

	\begin{definition}\label{iml} 
		A nonempty subset $I$ in $\la$ is called an \emph{ideal}, if for all $x,$ $y,$ $l\in \la$,
		\begin{enumerate}
			
			\item\label{cuj} $x,$ $ y \in I$ implies that $x\vee y\in I$,  and 
			
			\item\label{cul} $x\in I$ and $l\preccurlyeq x$ implies that $l\in I$.
		\end{enumerate}
		We shall denote the set of all ideals in $\la$ by $\id$, and by $0$, the zero ideal in $\la$. 
	\end{definition}

	\begin{remark}\label{iax}
		One may immediately notice that this definition does not involve any axiom related to the operation $\oplus$. If $l \preccurlyeq x$, where $l \in \la$ and $x \in I$, then according to Lemma \ref{bip}(\ref{mul}), we can deduce that $l\pr x \in I$. This implies that in the context of rings, the property of an ideal mentioned above is automatically satisfied in a quantale, whereas it needs to be explicitly defined as an axiom. In simpler terms, the definitions of an ideal in a complete lattice and in a quantale are identical. 
		
		Note that a definition of an ideal in a (commutative) quantale without identity  requires the following additional axiom: 
		\begin{equation}
			\label{exa} \text{If}\;\; l \preccurlyeq x,\; \text{where}\; l \in \la\; \text{and}\; x \in I,\; \text{and}\; l\pr x \in I,
		\end{equation}
		(see \cite[Definition 2.2]{YX13} and \cite[Definition 2(2)]{BAS16}). This is because we have obtained  (\ref{exa}) using the fact that $x\pr y\preccurlyeq x\wedge y$ holds for all $x, y\in \la$, which we proved assuming the existence of an identity element in $\la$ (see Lemma \ref{bip}(\ref{mul})).
	\end{remark}
	
	Based on Definition \ref{iml}, we shall now introduce a set of operations on ideals in a quantale. These operations  play a fundamental role in the development of our theory.

	\begin{definition}\label{opi}
		Let $\la$ be a quantale.
		\begin{enumerate}
			\item\label{meet}  The \emph{meet} $\bigwedge_{\lambda \in \Lambda} I_{\lambda}$ of ideals $\{I\}_{\lambda\in \Lambda}$  in $\la$ is given by their intersection.
			
			\item\label{join} The \emph{join} of ideals $\{I\}_{\lambda\in \Lambda}$  of $\la$ is defined by
			\begin{equation*}
				\bigvee_{\lambda \in \Lambda} I_{\lambda}=\left\{l\in \la\mid l\preccurlyeq \bigvee_{\text{finite}} x_{\lambda}\;\text{for}\; x_{\lambda}\in I_{\lambda}\right\}.
			\end{equation*}
			
			\item\label{prodc} The \emph{product} of two ideals $I$ and $J$ in $\la$ is defined as
			\begin{equation*}\label{prij}
				I\pr J=\left\{l\in \la\mid l\preccurlyeq \bigvee_{\text{finite}} \left\{x\pr y\mid x\in I, y\in J\right\}\right\}.
			\end{equation*}
			
			\item The \emph{residual} of an ideal
			$I$ by an ideal $J$ is defined by 
			\[(I:J)=\{x\in \la \mid x\pr J\preccurlyeq  I\}=\{x\in \la\mid x\pr j\in I,\;\text{for all}\;j\in J\}.\]
		\end{enumerate}
	\end{definition}

	Through the operations defined above, we can establish that the resulting entities are indeed ideals. This  be demonstrated in the following lemma.

	\begin{lemma}
		Let $\la$ be a quantale. If $I$ and $J$ are ideals in $\la$, then so are $(\mathrm{a})\;I\wedge J,$ $(\mathrm{b})\;I\vee J,$ $(\mathrm{c})\;I\pr J$, and $(\mathrm{d})\;(I:J)$.
	\end{lemma}

	\begin{proof}
		(a) Since $I$ and $J$ are ideals in $\la$,  if $x,$ $y\in I\wedge J,$ then  $x\vee y \in I\wedge J.$ Moreover, if $z\in I\wedge J$ then for any $x\leqslant z$, $x\in I\wedge J.$	Thus $I\wedge J$ is an ideal. 
		
		(b) To show $I\vee J$ is an ideal, let $l, l'\in I \vee J.$ This implies that $l\preccurlyeq x\vee y$ and $l'\preccurlyeq x'\vee y'$ for some $x,$ $ x'\in I$ and $y,$ $ y'\in J$. Therefore,
		\[l\vee l' \preccurlyeq (x\vee y)\vee (x'\vee y')=(x\vee x')\vee (y\vee y')\in I\vee J.\] Also, if $l\in I\vee J$ and $l'\preccurlyeq l$, then $l'\preccurlyeq l \preccurlyeq x\vee y$ for some $x\in I$ and $y\in J$. This implies that $l'\in I\vee J$. 
		
		(c) The proof of $I\pr J$ is an ideal is similar to (b). 
		
		(d) Finally, to show $(I:J)$ is an ideal, let $l,$ $ l'\in (I:J).$ Then $l\pr J\preccurlyeq  I$ and $l'\pr J\preccurlyeq  I$. This, by Definition \ref{mld}(\ref{jdp}) implies that \[(l\vee l')\pr J=(l\pr J)\vee (l'\pr J)\preccurlyeq  I.\] Suppose $l\in (I:J)$ and $l'\preccurlyeq l.$ Then by Lemma \ref{bip}(\ref{mon}), we have $l'\pr j\preccurlyeq l\pr j$ for all $j\in J$. Hence $l'\pr J\preccurlyeq  l\pr J\preccurlyeq  I$. In other words, $l'\in (I:J)$ as required.
	\end{proof}

	\begin{definition}\label{igs}
		If $S$ is a nonempty subset in a $\la$, then the \emph{ideal generated by} $S$ is defined by \[\langle S\rangle=\left\{x\in \la \mid x\preccurlyeq \bigvee_{i=1}^na_i,\;\text{for some}\;n\in \mathds{N}\;\text{and}\; a_i\in \la \pr S\right\}.\] 
	\end{definition}
	
	In particular, if $S=\{s\}$, then the ideal $\langle \{s\} \rangle$ is called \emph{principal}. We shall write $\langle s \rangle$ for $\langle \{s\} \rangle$. The following lemma is going to be useful in the sequel.

	\begin{lemma}\label{idss}
		Let $\la$ be a quantale.
		If $S$ and $T$ are nonempty subsets of a quantale $\la$, then $\langle S\wedge T\rangle=\langle S\rangle \wedge \langle T \rangle.$
	\end{lemma}
	
	\begin{proof}
		Since $S\wedge T\preccurlyeq S$ and $S\wedge T\preccurlyeq T,$ we immediately have $\langle S\wedge T\rangle \preccurlyeq \langle S\rangle \wedge \langle T \rangle.$ For the converse, let $x\in \langle S\rangle \wedge \langle T \rangle.$ Then $x\preccurlyeq \bigvee_{i=1}^n a_i$ and $x\preccurlyeq \bigvee_{j=1}^m b_i$, for some $a_i\in \la \pr S$ and $b_j\in \la\pr  T$, where $1\leqslant i\leqslant n$ and $1\leqslant j \leqslant m$. This implies that, \[x\preccurlyeq \left( \bigvee\limits_{i=1}^n a_i\right)\wedge \left(\bigvee\limits_{j=1}^m b_i \right),\] in other words, $x\in \langle S\wedge T\rangle.$
	\end{proof}

\begin{remark}
	In due course, we shall see results only with meet operations as in the above lemma, and with the appropriate terminologies and definitions,  similar results  can be proved for joins. 
\end{remark}

Before we proceed studying properties, let us see some examples of ideals in various quantales. Following a similar strategy, one can obtain examples of specific kind of ideals that we are going to discuss in this paper.

\begin{example}

	Let $X$ be a nonempty set and consider the power set $\mathcal{P}(X)$ ordered by inclusion $\subseteq$, with
	join as union,
	meet as intersection,
	top element as $X$, and bottom element as the $ \emptyset$.
	Then $\mathcal{P}(X)$ forms a commutative unital quantale.
	Let $X = \{1,2,3\}$. Then
	$
	I := \{ A \subseteq X \mid A \subseteq \{1\} \} = \{\emptyset, \{1\}\}
	$
	is an ideal in $\mathcal{P}(X)$.
\end{example}	

\begin{example}
	Let $X$ be a set, and define $Q := \mathcal{P}(X \times X)$ with join as union,
	multiplication as relational composition: $R \& S := \{ (x,z) \mid \exists y \in X,\; (x,y) \in R,\; (y,z) \in S \},$ unit as the identity relation $\{ (x,x) \mid x \in X \}$, and the bottom element as the $\emptyset$. It is easy to see that 	This is a (noncommutative) quantale. Let $X = \{1,2\}$, and let $R = \{(1,1), (1,2)\}$, $S = \{(2,1)\}$. Then
	$
	R \& S = \{ (1,1) \}.
	$
	Define
	$
	I := \{ R \subseteq X \times X \mid \text{no pair of form } (1,2) \text{ occurs in } R \}.
	$
	Then $I$ is an ideal.
\end{example}

\begin{example}
	Let $(P, \leqslant)$ be a poset. The set of lower sets
	$
	\mathcal{L}(P) := \{ L \subseteq P \mid x \in L,\, y \leq x \Rightarrow y \in L \}
	$
	is a complete lattice.
	Define multiplication by intersection: $A \& B := A \cap B$.
	Let $P = \mathds{N}$ with usual order. Then
	$
	I := \{ L \in \mathcal{L}(\mathds{N}) \mid \max(L) \leqslant 5 \}
	$
	is an ideal in $\mathcal{L}(\mathds{N})$.		
\end{example}

\begin{example}
	Let $X$ be a topological space. The set of open subsets $\mathcal{O}(X)$ forms a frame, and hence, a quantale.
	Let $X = \mathds{R}$ with the usual topology. Then the collection
	$
	I := \{ U \in \mathcal{O}(\mathbb{R}) \mid U \subseteq (-\infty, a) \text{ for some } a \in \mathbb{R} \}
	$
	is an ideal in $\mathcal{O}(\mathds{R})$.		
\end{example}

\begin{example}
	Consider $Q = [0,1]$ with usual order. Define
	join as $\max$,
	multiplication as $a \& b := ab$,
	unit as $1$, and the bottom element as $0$.		
	For fixed $r \in (0,1)$, define
	$
	I := [0,r] = \{ x \in [0,1] \mid x \leqslant r \}.
	$
	Then $I$ is an ideal in the quantale $([0,1], \leqslant, \max, \cdot)$.
\end{example}
	
	The algebraic manipulation of ideals in a quantale closely resembles that of a commutative ring. We can draw analogies between the following ideal-theoretic relations and their counterparts in the elementwise setting in a multiplicative lattice (see \cite{WD39, Dil62}).
	
	\begin{proposition}\label{bpi}
		Let $I$, $\{I_{\lambda}\}_{\lambda\in \Lambda},$ $J$, $\{J_{\lambda}\}_{\lambda\in \Lambda},$ and $K$ be ideals in a quantale $\la$. Then the following holds.
		\begin{enumerate}
			
			\item\label{idas} $I\pr (J\pr K)=(I\pr J)\pr K$.
			
			\item\label{idco} $I\pr J=J\pr I$. 
			
			\item\label{lii} $\la \pr I=I.$ 
			
			\item $0\pr I=0.$
			
			\item\label{ijd} $I\pr \left(\bigvee\limits_{\lambda \in \Lambda} J_{\lambda}\right)=\bigvee\limits_{\lambda \in \Lambda}(I\pr J_{\lambda}).$
			
			\item\label{mip} $I\pr J\preccurlyeq  I\wedge J$.
			
			\item\label{ijks} $I\pr (J\wedge K)\preccurlyeq (I\pr J)\wedge (I\pr K).$
			
			\item $(I\vee J)\pr (J\vee K)\preccurlyeq (I\pr J)\vee K.$
			
			\item If $I\vee K=J\vee K=\la$, then $((I\pr J)\vee K)=\la.$
			
			\item If $I\vee K=\la$, then $(I\wedge J)\vee K=J\vee K.$
			
			\item\label{ijji} $(I: J)\pr J\preccurlyeq  I.$
			
			\item\label{iicj} $I\preccurlyeq  (I:J).$
			
			\item $J\preccurlyeq I$ if and only if $(I : J)=\la$.
			
			\item $(I: \la)=I.$
			
			\item $I\preccurlyeq ((I\pr J) :J).$

			\item\label{inqi} $\left(\bigwedge\limits_{\lambda \in \Lambda} I_{\lambda}: J\right)=\bigwedge\limits_{\lambda \in \Lambda}(I_{\lambda}:J).$
			
			\item $\bigwedge\limits_{\lambda \in \Lambda}(I:J_{\lambda})\leqslant\left(I: \bigvee\limits_{\lambda \in \Lambda}J_{\lambda}\right).$
			
			\item $((I:J):K)=(I: (J\pr K)).$
			
			\item $(I: J)=(I:(I\vee J)).$
			
			\item $(I:J)=((I \wedge J): J).$
		\end{enumerate}
	\end{proposition}

	\begin{proof}
		(1) Follows from Definition \ref{mld}(\ref{mla}).
		
		(2) Follows from Definition \ref{mld}(\ref{mlc}).
		
		(3) From Remark \ref{iax}, it follows that $\la \pr I\preccurlyeq  I$. Indeed, $x\in \la,$ $i\in I$ implies $x\pr i\preccurlyeq x\wedge i \preccurlyeq i\in I$. By Definition \ref{mld}(\ref{mid}), $i=\top\pr i\in \la \pr I$ for all $i\in I$, proving that $I\preccurlyeq  \la\pr  I$.
		
		(4) Let $i\in I$. Then $\bot\pr i\preccurlyeq \bot\wedge i=\bot\preccurlyeq i$ implies that $\bot\pr i\in I$, and hence, $0\pr I\preccurlyeq  I$. Since $\bot\preccurlyeq x$ for all $x\in \la$, in particular, $\bot\preccurlyeq \bot\pr i$ for all $i\in I$, showing that $0\preccurlyeq  0\pr I$.
		
		(5) Follows from Definition \ref{mld}(\ref{jdp}).
		
		(6) Follows from Remark \ref{iax}.
		
		(7) Since $J\wedge K\preccurlyeq  J$ and  $J\wedge K\preccurlyeq  K,$ by Lemma \ref{bip}(\ref{mon}), we obtain $I\pr (J\wedge K)\preccurlyeq (I\pr J)\wedge (I\pr K).$
		
		(8) For all $i\in I,$ $j\in J$, and $k\in K$, applying Definition \ref{mld}(\ref{jdp}), we have \[(i\vee k)\pr (j\vee k)=(i\pr j)\vee (i\vee j \vee k)\pr k\in (I\pr J)\vee K.\]
		
		(9) From (8), it follows that $\la \preccurlyeq ((I\pr J)\vee K),$ whereas the other inclusion is trivial.
		
		(10) Since $I\wedge J\preccurlyeq  J$, we have $(I\wedge J)\vee K\preccurlyeq  J\vee K$. For the other half of the inclusion, we have
		\begin{align*}
			J\vee K&=(I\vee K)\pr (J\vee K)\\& =I\pr (J\vee K)\vee K\pr (J\vee K)\\& \preccurlyeq I\pr (J\vee K)\vee K\\& =(I\pr J)\vee ((I\pr K)\vee K)\\& =(I\pr J)\vee K\\& \preccurlyeq (I\wedge J)\vee K.	
		\end{align*}
		
		(11) If $x\in (I:J)\pr J,$ then for some finite set $\{1, \ldots, n\}$, we have $x\preccurlyeq \bigvee_{i=1}^n l_i\pr j_i\in J$, where $l_i\in \la$ and $j_i\in J$.
		
		(12) If $i\in I$, then by Lemma \ref{bip}(\ref{mon}), $i\pr J\preccurlyeq  I\pr J$, and since $I$ is an ideal, by Remark \ref{iax}, $i\pr J\preccurlyeq  I$.
		
		(13) If $J\preccurlyeq I$ then by applying (3),we have $\la\pr  J=J\preccurlyeq  I$. Hence $(I:J)=\la$. Conversely, if $(I:J)=\la$, then by Definition \ref{mld}(\ref{mid}), $\top\pr j=j\in I$ for all $j\in J$.
		
		(14) Follows from (3).
		
		(15) Follows by Lemma \ref{bip}(\ref{mon}).
		
		(16) The desired identity follows from the following chain of equivalent statements:
		\[l\in \left(\bigwedge_{\lambda \in \Lambda} I_{\lambda}: J\right) \Leftrightarrow l\pr J\preccurlyeq  \bigwedge_{\lambda \in \Lambda} I_{\lambda}\Leftrightarrow l\pr J\preccurlyeq  I_{\lambda}, \forall \lambda \in \Lambda\Leftrightarrow l\in \bigwedge_{\lambda \in \Lambda}(I_{\lambda}:J).\]
		
		(17) Note that
		\begin{align*}
			l\in \bigwedge_{\lambda \in \Lambda}(I:J_{\lambda}) \Rightarrow l\pr J_{\lambda}\subseteq I,\forall \lambda \in \Lambda\Rightarrow \bigvee_{\lambda \in \Lambda} l\pr J_{\lambda}\subseteq I\Rightarrow l\pr \left(  \bigvee_{\lambda \in \Lambda} J_{\lambda}\right)\subseteq I\Rightarrow l\in (I: \bigvee_{\lambda \in \Lambda} J_{\lambda}).\end{align*}
		
		(18) Observe that $l\in ((I:J):K)\Leftrightarrow l\pr K\preccurlyeq  (I:J)\Leftrightarrow (l\pr K)\pr J\preccurlyeq  I\Leftrightarrow l\pr ((J\pr K))\subseteq I\Leftrightarrow l\in (I:J\pr K).$
		
		(19) Note that 
		\begin{align*}
			x\in (I:J)\Rightarrow x\pr J\preccurlyeq  I\Rightarrow x\pr I\vee x\pr J\preccurlyeq  x\pr I\vee I\Rightarrow x\pr (I\vee J)\preccurlyeq I\Rightarrow x\in (I: (I\vee J)),
		\end{align*}
		and for the other half of the inclusion, 
		\begin{align*}
			x\in (I:(I\vee J))\Rightarrow x\pr (I\vee J)\preccurlyeq I\Rightarrow x\pr I\vee x\pr J\preccurlyeq  I\Rightarrow x\pr J\preccurlyeq  I.
		\end{align*}
		
		(20) Finally, we have $x\in (I:J)\Leftrightarrow x\pr J\preccurlyeq  I \Leftrightarrow x\pr J\preccurlyeq  I\wedge J \Leftrightarrow x\in ((I\wedge J):J).$
	\end{proof}
	
	The well-known fact is that the set of ideals in a commutative ring with identity forms a quantale, following the definition provided in Definition \ref{mld}. Now, we shall proceed to demonstrate in the following theorem that the set of ideals in a quantale possesses an analogous structure.

	\begin{theorem}
		If $\la$ is a quantale, then $(\id, \preccurlyeq, \vee, \wedge, \bot, \top, \pr )$ is a quantale.
	\end{theorem}
	
	\begin{proof}
		By (\ref{meet}) and (\ref{join}) of Definition \ref{opi}, it follows that arbitrary joins and arbitrary meets exist in the poset $(\id, \preccurlyeq),$ where $\preccurlyeq$ is the usual subset inclusion relation. It is easy to see that the bottom element $\bot$ and the top element $\top$ of the poset $(\id, \preccurlyeq)$ are respectively $0$ (the zero ideal) and $\la$. Hence, $(\id, \preccurlyeq)$ is a complete lattice. Taking the multiplication operation between ideals in $\la$ as in Definition \ref{opi}(\ref{prodc}), it  follows respectively from  (\ref{idas}), (\ref{idco}), (\ref{ijd}), and (\ref{lii}) of Proposition \ref{bpi} that the operation $\pr $ satisfies Axioms (\ref{mla}), (\ref{mlc}), (\ref{jdp}), and (\ref{mid}) of Definition \ref{mld}.
	\end{proof}
	
	Similar to rings, the concept of an annihilator of a subset in a quantale is also established as a special case of the residual operation.

	\begin{definition}
		If $S$ is a nonempty subset in a quantale $\la$, then the \emph{annihilator of $S$} is defined by \[\an(S)=\{x\in \la \mid x\pr s=\bot,\;\text{for all}\;s\in S\}.\]	
	\end{definition}
	
	In the following lemma, we discuss a few elementary properties of annihilators.

	\begin{lemma}
		Let $S$ and $T$ be subsets in a quantale $\la$. Then the following holds.
		
		\begin{enumerate}
			\item $S\preccurlyeq  T$ implies that $\an(T)\preccurlyeq \an(S).$
			
			\item $S\preccurlyeq  \an(\an(S)).$
			
			\item $\an(S)=\an(\an(\an(S))).$
		\end{enumerate}
	\end{lemma}

	\begin{proof}
		(1) If $x\in \an(T),$ then $x\pr t=\bot$ for all $t\in T$. Since $S\preccurlyeq T,$ this implies that $x\pr s=\bot$ for all $s\in S$. Hence $\an(T)\preccurlyeq \an(S).$			
		
		(2) If $s\in S$, then $x\pr s=s\pr x=\bot$ for all $x\in \an(S),$ and hence $s\in \an(\an(S)).$ 
		
		(3) Since by (2), $S\preccurlyeq \an(\an(S))$, by (1), it follows that $\an(\an(\an(S)))\preccurlyeq \an(S).$ By (2), it follows that $\an(S)\preccurlyeq \an(\an(\an(S))).$
	\end{proof}

	\begin{definition}
		Let $\la$ be a quantale.
		\begin{enumerate}
			\item An ideal $I$ of $\la$ is called \emph{proper} if $I\neq \la$, and we shall denote the set of all proper ideals in $\la$ by $\idp$.
			
			\item A proper ideal $M$ of $\la$ is called \emph{maximal} if there is no other proper ideal $I$ of $\la$ properly containing $M$. A quantale $\la$ with exactly one maximal ideal is called \emph{local}.
			
			\item An ideal $I$ of $\la$ is called \emph{minimal} if $I\neq 0$ and $I$ properly contains no other nonzero ideals in $\la$.
		\end{enumerate}
	\end{definition}
	
	The subsequent result ensures that the set of maximal ideals in a quantale is not empty.

	\begin{theorem}\label{emi}
		Every quantale $\la$ with $0\neq 1$ has a maximal ideal. 
	\end{theorem}

	\begin{proof}
		Since $0\in \idp$, $\mathrm{Prp}(\la)\neq \emptyset.$ Consider a chain $\{I_{\lambda}\}_{\lambda\in \Lambda}$ of ideals in $\idp$. By Definition \ref{opi}(\ref{meet}), it follows that $\bigvee_{ \lambda \in \Lambda} I_{\lambda}\in \idp$ and $I_{\lambda}\preccurlyeq\bigvee_{ \lambda \in \Lambda} I_{\lambda}$ for all $\lambda \in \Lambda$. Hence, by Zorn's lemma, $\mathrm{Prp}(\la)$ has a maximal element, which is our desired maximal ideal in $\la$.
	\end{proof}
	
	\begin{corollary}\label{eicm}
		Every proper ideal in a quantale $\la$ is contained in a maximal ideal in $\la$.
	\end{corollary}
	
	The following proposition presents a condition that is sufficient for a quantale to be local.

	\begin{proposition}
		Let $\la$ be a quantale and $M\neq \la$ an ideal in $\la$ such that every $x\in \la\neg M$ is a unit in $\la$. Then $\la$ is  local  and $M$ is its maximal ideal.
	\end{proposition}

	\begin{proof}
		Let $I$ be an ideal in $\la$ such that $M\subsetneq I\preccurlyeq  \la.$ By hypothesis, every element $x\in I\neg M$ is a unit and since $M\neq \la$, we must have $\top=x\pr y\in I$ for some $y\in \la$. Hence $I=\la$, showing that $M$ is maximal. Since every element of $\la\neg M$ is a unit, if $M'$ be another maximal ideal in $\la$, then we must have \[(\la\neg M)\wedge M'=\emptyset,\] which implies that $M'\preccurlyeq M$. But $M'$ is a maximal ideal. Therefore, $M'=M.$
	\end{proof}

	Our next objective is to examine the properties of ideals and their products, meets, and residual operations under quantale homomorphisms. These properties serve as the lattice-theoretic counterparts to their respective ring-theoretic versions (see \cite[Proposition 1.17 and Exercise 1.18]{AM69}).
	To this end, recall that a map $\phi\colon \la \to \la'$ from $\la$ to $\la'$ is called a \emph{quantale homomorphism} if 
	\begin{enumerate}
		\item  $x\preccurlyeq x'$ implies that $\phi(x)\preccurlyeq \phi(x');$	
		
		\item  $\phi(x\vee x')=\phi(l)\vee \phi(l');$
		
		\item  $\phi(x\wedge x')=\phi(x)\wedge \phi(x');$
		
		\item  $\phi(x\pr  x')=\phi(x)\pr  \phi(x'),$
	\end{enumerate}
	for all $x,$ $x'\in \la$. 
	Suppose that  $\phi\colon \la \to \la'$ is a quantale homomorphism.
	If $J$ is an ideal in $\la'$, then the \emph{contraction of} $J$, denoted by $J^c$, is defined by $\phi\inv (J).$
	If $I$ is an ideal in  $\la$, then the \emph{extension of} $I$, denoted by $I^e$, is defined by $\langle \phi(I)\rangle.$

	\begin{theorem}\label{cep}
		
		Let $\phi\colon \la \to \la'$ be a quantale homomorphism. For $I,$ $I_1$, $I_2\in \id$, and  $J,$ $J_1$, $J_2\in\mathcal{I}(\la')$, the following hold.
		
		\begin{enumerate}
			
			\item\label{jcki}  $J^c$ is an ideal in  $\la$. 
			
			\item $I^e$ is an ideal in $\la'$.
			
			\item\label{ijec} $\mathrm{(a)}\; I\preccurlyeq  I^{ec}$. $\mathrm{(b)}\;  J^{ce}\preccurlyeq J$. $\mathrm{(c)}\; J^c=J^{cec}$. $\mathrm{(d)}\; I^e=I^{ece}.$
			
			\item\label{bij} There is a bijection between the sets $\{I\mid I^{ec}=I\}$ and $\{J\mid J^{ce}=J\}$.
			
			\item\label{dri} $\mathrm{(a)}\; (I_1\wedge  I_2)^e\preccurlyeq  I_1^e\wedge  I_2^e$. $\mathrm{(b)}\; (I_1\pr I_2)^e=I_1^e\pr I_2^e$. $\mathrm{(c)}\; (I_1:I_2)^e\preccurlyeq  (I_1^e:I_2^e)$. 
			
			\item\label{crj} $\mathrm{(a)}\; (J_1\wedge  J_2)^c= J_1^c\wedge  J_2^c$. $\mathrm{(b)}\; J_1^c\pr J_2^c \preccurlyeq (J_1\pr J_2)^c$. $\mathrm{(c)}\; (J_1:J_2)^c\preccurlyeq  (J_1^c:J_2^c)$. 
		\end{enumerate}	
	\end{theorem}

	\begin{proof}
		(1)	Let $x,$ $x'\in J^c$. Then $\phi(x),$ $\phi(x')\in J$. Since $J$ is an ideal, we have $\phi(x\vee x')=\phi(x)\vee \phi(x')\in J,$ implying that $x\vee x'\in J^c$. If $x'\preccurlyeq x$ and $x\in J^c$, then $\phi(x')\preccurlyeq \phi(x)\in J$, proving that $x'\in J^c$, as required.
		
		(2) Follows from Definition \ref{igs}.
		
		(3)--(4) Follows by usual set-theoretic arguments.
		
		5(a) Observe that \[(I_1\wedge I_2)^e=\langle \phi(I_1\wedge I_2)\rangle \preccurlyeq \langle \phi(I_1)\wedge \phi(I_2)\rangle=\langle \phi(I_1)\rangle \wedge \langle \phi(I_2)\rangle=I_1^e\wedge I_2^e,\] where, the second equality follows from Lemma \ref{idss}.
		
		5(b) First, we show that $\langle \phi(I_1\pr I_2)\rangle=\langle \phi(I_1)\rangle \pr \langle \phi(I_2)\rangle.$ Since $I_1\pr I_2\preccurlyeq I_1$ and $I_1\pr I_2\preccurlyeq I_2,$ we immediately obtain that $\langle \phi(I_1\pr I_2)\rangle\preccurlyeq \langle \phi(I_1)\rangle \pr \langle \phi(I_2)\rangle.$ For the other inclusion, let $x\in \langle \phi(I_1)\rangle \pr \langle \phi(I_2)\rangle.$ This implies $x=x_1\pr x_2$ for some $x_1,$ $x_2\in \la$ such that $x_1\preccurlyeq \bigvee_{i=1}^n a_i$ and $x_2\preccurlyeq \bigvee_{j=1}^m b_j,$ where $a_i\in \la \pr \phi(I_1)$ and $b_j\in \la \pr \phi(I_2)$ for all $1\leqslant i \leqslant n$ and $1\leqslant j \leqslant m$. So, \[x=x_1\pr x_2\preccurlyeq \left( \bigvee_{i=1}^n a_i\right)\pr \left( \bigvee_{j=1}^m b_j\right)=\bigvee_{i=1}^n\bigvee_{j=1}^m a_i\pr b_j, \]
		where \[a_i\pr b_j\in \la\pr \phi(I_1)\pr \la\pr \phi(I_2)=\la\la\pr \phi(I_1)\pr\phi(I_2)=\la\pr \phi(I_1)\pr\phi(I_2).\] Therefore, $x\in \langle \phi(I_1\pr I_2)\rangle.$ Now, $(I_1\pr I_2)^e=\langle \phi(I_1\pr I_2)\rangle =\langle \phi(I_1)\rangle \pr \langle \phi(I_2)\rangle=I_1^e\pr I_2^e.$
		
		5(c) By Propositions \ref{bpi}(\ref{ijji}) and 5(b), we obtain \[(I_1:I_2)^e\pr I_2^e=((I_1:I_2)\pr I_2)^e\preccurlyeq I_1^e. \]	
		
		6. The proofs are similar to (5).
	\end{proof}

	\subsection{Prime and semiprime ideals}\label{psi}
	
	In this section, our objective is to present the concept of prime and semiprime ideals in the context of a quantale. We shall demonstrate that their definitions based on  elements are equivalent to the definitions rooted in ideals. Additionally, we shall explore the topic of radical ideals and delve into some of their basic properties. Some of these properties are direct generalizations of their ring-theoretic counterparts (see \cite{CZ20}). 
	
	\begin{definition}
		Let $\la$ be a quantale.
		\begin{enumerate}
			\item A proper ideal $P$ of  $\la$ is called \emph{prime} if $x\pr y\in P$ implies $x\in P$ or $y\in P$ for all $x,$ $y\in \la$. By $\spc$, we denote the set of all prime ideals in $\la$. An ideal $P$ of $\la$ is called \emph{minimal prime} if $P$ is both a minimal ideal and a prime ideal.
			
			\item A proper ideal $P$ of $\la$ is called \emph{semiprime} if $x^2 \in P$  implies that $x\in P$ for all $x\in \la$.
			
			\item Two ideals $I$ and $J$ of $\la$ is said to be \emph{coprime}  if $I\vee J = \la$.	
			
			\item $\la$ is called \emph{Noetherian} if every ascending chain of ideals in $\la$ is eventually stationary.
		\end{enumerate}
	\end{definition}

\begin{proposition}\label{lpsp}
		Let $\la$ be a quantale.
		\begin{enumerate}
			
			\item\label{kpp} An ideal $P\in \spc$ if and only if $I\pr J\preccurlyeq   P$ implies $I\preccurlyeq  P$ or $J\preccurlyeq  P,$ for all $I, J\in \id$.
			
			\item\label{mpiq} Any prime ideal $P$ of $\la$ contains
			a prime ideal $P'$  such that $S\preccurlyeq P'$.
			
			\item \label{wnknp} 
			If $I$ is a proper ideal in a Noetherian quantale $\la$, then $\la$ has only a finite number of
			minimal prime ideals over $I$.
			
			\item An ideal $I$ of $\la$ is semiprime if and only if $J\pr J\preccurlyeq  I$ implies $J\preccurlyeq  I,$ for all $J\in \id.$
			
			\item Let $\la$ be a quantale. If $I$ and $J$ are coprime ideals in $\la$, then $I\wedge J=IJ.$ 
		\end{enumerate}
	\end{proposition}

	\begin{proof}
		(1) Suppose $P$ is a  prime ideal, and $I\pr J\preccurlyeq  P$ with $J\neq P$. This implies the existence of an element $j\in J$ such that $j\notin P$. For any $i\in I$, we then have, $i\pr j\in P$ and based on the assumption, this implies $i\in P$. Since $i$ was chosen arbitrarily from $I$, we conclude that $I\preccurlyeq  P$, as required. 
		Conversely, suppose $I\pr J\preccurlyeq  P$ implies either $I\preccurlyeq  P$ or $J\preccurlyeq  P$. Now consider $x\pr y\in P$ for some $x,$ $y\in \la$. Let $a\pr b\in \langle x\rangle \pr  \langle y \rangle,$ where $a\in \langle x\rangle$ and $b\in \langle y\rangle$. This implies that \[a\pr b\preccurlyeq \left(\bigvee_{i=1}^n x_i\right) \pr \left(\bigvee_{j=1}^m y_j\right)=\bigvee_{i=1}^n\left(\bigvee_{j=1}^m (x_i\pr y_j)\right),\] where $x_i\pr y_j\in (x\pr y)\la,$ for all $i=1, \ldots, n$ and $j=1, \ldots, m.$ Thus, $x_i\pr y_j\in P$, and hence $a\pr b\in P$, implying that $\langle x\rangle \pr \langle y \rangle\preccurlyeq  P.$ By the assumption, this implies either $\langle x\rangle\preccurlyeq  P$ or $\langle y\rangle\preccurlyeq  P;$ in other words, either $x\in P$ or $y\in P$.
		
		(2) Suppose $\Omega=\{P'\in \spc \mid S\preccurlyeq  P' \preccurlyeq P\}.$ Since $P\in \Omega$, the set $\Omega$ is nonempty. Consider a subset $\{P'_{\lambda}\}_{\lambda \in \Lambda}$  of  decreasing chain of prime ideals 
		of $\Omega$. Then by Zorn's lemma, there exists a minimal element $\bigwedge_{\lambda \in \Lambda}P'_{\lambda}$ of that chain, and this minimal element is our desired prime ideal.
		
		(3)
		First we show that every ideal containing $I$ contains a finite product of prime ideals, each
		containing $I$. Suppose that this is not the case, and let $S$ be the set of ideals containing $I$ which
		do not contain a finite product of prime ideals each containing $I$. By hypothesis, $S$ is not empty.
		Let $C$ be a chain in $S$. As $\la$ is Noetherian, $C$ has a maximum element. From Zorn's lemma, $S$ has a
		maximal element $M$. As $\la$ has a prime ideal containing $I$, $M \neq\la$. Also, $M$ is not prime.
		There exist $a$, $b \in \la$ such that $a\pr b \in M$ and $a \notin M$, $b \notin M$. Setting \[A = \langle M, a\rangle, \quad
		B = \langle M, b\rangle,\] we obtain $A\pr B \preccurlyeq M$, with $A$ and $B$ strictly included in $M$. Since $M$ is maximal, $A$
		and $B$ both contain a finite product of prime ideals each containing $I$, hence so does $M$, which
		contradicts the fact that $M\in S$. It follows that every ideal containing I contains a finite product
		of prime ideals, each containing $I$.
		We now apply this result to the ideal $I$: there exist prime ideals $P_1,\ldots, P_n$, each containing
		$I$, whose product is contained in $I$. We claim that any minimal prime $P$ over $I$ is among the
		$P_i$. Indeed, \[P_1\pr\cdots\pr P_n\preccurlyeq I \preccurlyeq P.\] We deduce that $P_i\preccurlyeq P$, for some $i$. However, $P$ is minimal, so
		$P_i = P$, and it follows that there is only a finite number of minimal prime ideals over $I$.
		
		(4) Suppose $I$ is a semiprime ideal in $\la$. Let $j\in J$. Then $j^2\in J\pr J\preccurlyeq  I$. This implies that $j^2\in I$ and since $I$ is semirpime, $j\in I$, \textit{i.e.}, $J\preccurlyeq I$. For the converse, we first show that $\langle x \rangle\pr  \langle x \rangle\preccurlyeq \langle x^2\rangle$, for any $x\in \la$. Suppose $l\in \langle x \rangle\pr  \langle x \rangle.$ In particular, this implies that $l\preccurlyeq x^2,$ and hence $l\in \langle x^2\rangle.$ Now, let $x^2\in I$. Then \[\langle x \rangle\pr  \langle x \rangle\preccurlyeq\langle x^2\rangle \preccurlyeq I.\] By assumption, this means $x\in \langle x\rangle \preccurlyeq I.$ Hence, $I$ is a semiprime ideal.
		
		(5) Notice that \[(I\vee J)\pr (I\wedge J)=I\pr (I\wedge J)\vee J(I\wedge J)\preccurlyeq (I\wedge (I\pr J))\vee ((I\pr J)\wedge J)\preccurlyeq I\pr J,\] where the equality is obtained by Definition \ref{mld}(\ref{jdp}),  and the first inclusion follows from Proposition \ref{bpi}(\ref{ijks}). If $I\vee J=\la,$ then by Proposition \ref{bpi}(\ref{lii}), $\la(I\wedge J)=I\wedge J\preccurlyeq  I\pr J$.	
	\end{proof} 
	
	The following proposition is adapted from the realm of rings (see \cite[Lemma 3.19]{AK13}).

	\begin{proposition}[Prime avoidance lemma]
		Let $I$ be a subset in a quantale $\la$ that is stable under join and the operation $\pr$. Let  $P_1, \ldots, P_n$ be ideals in $\la$ such that $P_3, \ldots, P_n\in \spc$. If $I\neq P_j$ for all $j,$ then there is an $x \in I$ such that $x \notin
		P_j$ for all $j$.
	\end{proposition}

	\begin{proof}
		We prove the claim by induction. If $n=1$, the claim is trivially true. Suppose $n\geqslant 2$ and suppose for every $i$, there exists an $x_i\in I$ such that $x_i\notin P_j$ for all $j\neq i$. Assume that $x_i\in P_i$ for all $i$. If $n=2$, then $x_1\vee x_2\notin P_j$ for $j=1,$ $2$. Indeed, $x_2\preccurlyeq x_1\vee x_2\in P_1$ implies $x_2\in P_1,$ a contradiction, and similarly, $x_1\preccurlyeq x_1\vee x_2\in P_2$ implies $x_1\in P_2$, a contradiction.	For $n\geqslant 3$, \[(x_1\pr \cdots\pr x_{n-1})\vee x_n\notin P_j.\] For the case, $n=j,$ we have $x_n\in P_n$ and since $P_n$ is a prime ideal, $x_k\in P$ for some $k\in \{1, \ldots, n-1\},$ a contradiction. For, $j<n$, $x_n\in P_j$, again leads to a contradiction.
	\end{proof}
	
	Our following aim is to explore the radicals of ideals and their associations with prime and semiprime ideals.

	\begin{definition}
		The \emph{radical} of an ideal $I$ in a quantale $\la$ is defined as follows:
		\[\rd(I)=\{x\in \la\mid x^n\in I,\;\text{for some}\;n\in \mathds{N}\}.\]	
		An ideal $I$ is said to be \emph{radical ideal} if $\rd(I)=I.$ 
	\end{definition}

	The following proposition gives an equivalent definition of radical of an ideal, and it extends Proposition 1.14 of \cite{AM69}.

	\begin{proposition}\label{edr}
		If $I$ is an ideal in a quantale $\la$, then $\rd(I)=\bigwedge_{ P}\{P\in \spc\mid  I \preccurlyeq P\}.$
	\end{proposition}

	\begin{proof}
		Let $l\in \rd(I).$ Then there exists  $n\in \mathds{N}$ such that $l^n\in I$, and also $l^n\in P,$ for all $P\in \spc$ with $I\preccurlyeq P$. This implies that $l\in P$ for all such $P\in \spc.$ Hence, $l\in \bigwedge_{ P}\left\{P\in \spc\mid  I \preccurlyeq P\right\}.$ Conversely, assume that $l \notin \rd(I),$ for some $l\in \la$. Consider the set 
		\[\Omega=\{ J\in \id\mid I\preccurlyeq J \;\text{and}\; l^n\notin J, \forall n\in \mathds{N}\}.\]
		It is easy to see that $\Omega$ is nonempty and by Zorn's lemma, $(\Omega, \leqslant)$ has a maximal element, say $P$ such that $I\preccurlyeq P$. It suffices to show that $P$ is a prime ideal. Suppose $x,$ $y\notin  P$ for some $x,$ $y\in \la$. This implies, $\langle x, P\rangle\notin \Omega$ $\langle x, P\rangle\notin \Omega,$ which followingly implies $\langle x\pr y, P\rangle \notin \Omega$. Hence, $x\pr y\notin P$.
	\end{proof}
	
	In the next lemma, we compile some elementary properties of the radical of an ideal in a quantale.

	\begin{lemma}\label{radpr}
		For any ideals $I$, $J$, $\{I_{\lambda}\}_{\lambda\in \Lambda}$ in a quantale $\la$, the following hold.
		\begin{enumerate}
			\item\label{irad} $\rd(I)$ is an ideal containing $I$.	
			
			\item\label{ijrirj} If $I\preccurlyeq  J$, then $\rd(I)\preccurlyeq \rd(J).$
			
			\item $\rd(\rd(I)) =\rd(I).$
			
			\item $\rd(I)=\rd(I\pr \cdots \pr I)$ $($repeated $n$-times$)$.
			
			\item\label{ijicj} $\rd(I\wedge J)=\rd(I)\wedge \rd(J)=\rd(I\pr J).$
			
			\item $\bigvee_{\lambda \in \Lambda} \rd(I_{\lambda})\preccurlyeq \rd \left(\bigvee_{\lambda \in \Lambda} I_{\lambda}\right).$
			
			\item $\rd(I)=\la$ if and only if $I=\la.$
			
			\item $\rd(I\vee J)=\rd(\rd(I)\vee \rd(J)).$
		\end{enumerate}
	\end{lemma}

	\begin{proof} 
		(1)	 By taking $n=1$, it follows that $I\preccurlyeq \rd(I).$ To show $\rd(I)$ is an ideal, let $x\in \rd(I)$ and let $y\preccurlyeq x$, for some $y\in \la$. Then $x^n\in I$ for some $n\in \mathds{N}.$  By Lemma \ref{bip}(\ref{mon}), this implies $y^n\preccurlyeq x^n$, and since $I$ is an ideal, we obtain that $y^n\in I$. Hence $y\in \rd(I),$ showing that the condition (\ref{cul}) of Definition \ref{iml} holds. To check condition (\ref{cuj}) of Definition \ref{iml}, let $x,$ $y\in \rd(I)$. Then $x^n,$ $y^m\in \rd(I)$ for some $n$, $m\in \mathds{N}$. It suffices to show that $(x\vee y)^{m+n}\in I.$ Applying formula (\ref{bif}), we obtain
		\[(x\vee y)^{m+n}=\bigvee_{k=0}^{m+n} \binom{m+n}{k}\,x^{m+n-k}\pr y^k.\]
		Since $0\leqslant k \leqslant m+n$, we have two possibilities: either $n\leqslant i$, in which case, $x^i\in I$ (since $x^n\in I$); or $m\leqslant n+m-i,$ in which case $y^{n+m-i}\in I$  (since $y^m\in I$). This proves that $(x\vee y)^{m+n}\in I.$
		
		(2) Follows from Proposition \ref{edr}.
		
		(3) Applying (2) on  $I\preccurlyeq \rd(I)$ (which follows from (1)) gives $\rd(I)\preccurlyeq \rd(\rd(I)).$ Conversely, suppose that $x\in \rd(\rd(I)).$ This implies $x^n\in \rd(I)$ for some $n\in \mathds{N}$, which further implies that $(x^n)^m\in I,$ for some $m\in \mathds{N}$. However, $(x^n)^m=x^{nm}\in I.$ Hence $x\in \rd(I).$
		
		(4) Since $I\pr \cdots \pr I\preccurlyeq I$, by (2), we have $\rd(I\pr \cdots \pr I)\preccurlyeq \rd(I).$ If $x\in \rd(I)$, then $x^m\in I$, for some $m\in \mathds{N}$. But then $x^{mn}=(x^m)^n\in I\pr \cdots \pr I,$ implying that $x\in \rd (I\pr \cdots \pr I).$
		
		(5) For the first equality, $\rd(I\wedge J)\preccurlyeq \rd(I)\wedge \rd(J)$ follows from (2). If $x\in \rd(I)\wedge \rd(J),$ then $x^n\in \rd(I)$ and $x^m\in \rd(J)$ for some $n,$ $m\in \mathds{N}$. Let $k=\max\{n,m\}.$ Then $x^k\in I\wedge J,$ and hence $x\in \rd(I\wedge J).$ For the second equality, note that \[(I\wedge J)\pr(I\wedge J)\preccurlyeq I\pr J\preccurlyeq I\wedge J.\]
		Therefore, by (2) and (4), we obtain
		\[\rd((I\wedge J)\pr (I\wedge J))=\rd(I\wedge J)\preccurlyeq \rd(I\pr J)\preccurlyeq \rd (I\wedge J).\]
		
		(6) Follows from (2).
		
		(7) Follows from Proposition \ref{edr}.	
		
		(8) By (2) and (3), it follows that \[\rd(\rd(I)\vee \rd(J))\preccurlyeq \rd(\rd(I\vee J))=\rd(I\vee J).\] Since by (1), $I\preccurlyeq \rd(I)$ and $J\preccurlyeq \rd(J),$ we obtain $I\vee J\preccurlyeq \rd(I)\vee \rd(J),$ and by applying (2), we get the desired inclusion.	
	\end{proof}
	
	Expanding upon the previously introduced definition of the radical of an ideal in a quantale, we shall now introduce two additional types of radicals specific to quantales: nilradicals and Jacobson radicals. Additionally, we shall explore the relationships that exist among these distinct types of radicals.

	\begin{definition}
		Let $\la$ be a quantale.
		\begin{enumerate}
			\item An element $x$ of $\la$ is called \emph{nilpotent} if $x^n=\bot$ for some positive integer $n$. The set $\nl(\la)$ of all nilpotent elements of a quantale $\la$ is called the \emph{nilradical} of $\la$. A quantale $\la$ is called \emph{reduced} if $\nl(\la)=0$
			
			\item The \emph{Jacobson radical} $\mathcal{J}(\la)$ of a quantale $\la$ is defined as the intersection of all maximal ideals in $\la$.
			
			\item A \emph{zero-divisor} of $\la$ is an element $x$ of $\la$ for which there exists an element $y$ ($\neq \bot$) in $\la$ such that $x\pr y=\bot$. 
			If $\bot\neq \top$ in $\la$ and if $\la$ does not have any nonzero zero-divisors, then $\la$ is called a \emph{quantale domain} (QD). Note that a quantale $\la$ is a QD if and only if $0$ is a prime ideal.
		\end{enumerate}	
	\end{definition}

	\begin{proposition}
		Let $\la$ be a quantale. Then the following holds.
		\begin{enumerate}
			\item 
			$\nl(\la)\in \id$ and it is the intersection of prime ideals in $\la$. Moreover, $\nl(\la)\preccurlyeq \mathcal{J}(\la).$ 
			

			
			\item 
			$\la$ is reduced and has only one minimal prime ideal
			if and only if $\la$ is a QD.
		\end{enumerate}	
	\end{proposition}

	\begin{proof}
		(1) First we show $\nl(\la)$ is an ideal in $\la$. Suppose $x\in \nl(\la)$ and $y\preccurlyeq x$. Then $x^n=0$ for some $n\in \mathds{N}$ and by Lemma \ref{bip}(\ref{monj}), $y^n\preccurlyeq x^n=0$, implying that $y^n=0$, and hence, $y\in \nl(\la).$ Now, let $x,$ $y\in \nl(\la)$. This implies $x^n=0$ and $y^m=0$ for some $n,$ $m\in \mathds{N}$. Observe that  $(x\vee y)^{m+n-1}$ is the join of integer multiples of elements $x^r\pr y^s$ (see proof of Lemma \ref{radpr}(\ref{irad})), where $r+s=m+n-1$. Since we cannot have both $r<m$ and $s<n$, each of these terms vanishes, and hence \[(x\vee y)^{m+n-1}=0.\] The fact that $\nl(\la)$ is the intersection of prime ideals in $\la$ follows from Proposition \ref{edr}.
		Suppose $x\in \nl(\la).$ Then there exists $n\geqslant 1$ such that $x^n=\bot$. Let $M$ be a maximal ideal in $\la$. Then $x^n\in M$ and $M$ is prime, $x\in M$, proving that $\nl(\la)\preccurlyeq \mathcal{J}(\la)$.
		
		(2) Suppose $\la$ is reduced. Then $\nl(\la)=0$. By Proposition \ref{edr}, this implies \[0=\bigwedge_{ P\in \spc} P.\] Then by Lemma \ref{lpsp}(\ref{mpiq}) and hypothesis, there exists a prime ideal $P$ such that $P=0$. Hence $\la$ is a QD. The converse follows from the fact that $\la$ is a QD if and only if $0$ is a prime ideal in $\la$. 
	\end{proof}

	Our next objective is to establish the equivalence between semiprime ideals and radical ideals within a quantale. This equivalence is well-known in the realm of (noncommutative) rings. However, in the noncommutative scenario, the concepts of $m$-systems and $n$-systems in rings are necessary. On the other hand, when dealing with ideals in quantales (or commutative rings), the presence of multiplicatively closed subsets alone suffices. To demonstrate this equivalence, we shall first proceed by presenting a series of lemmas.

	\begin{lemma}\label{rpms}
		An ideal $P$ in a quantale $\la$ is prime if and only if the complement of $P$ in $\la$ forms a multiplicatively closed subset of $\la$. Here, a subset $S$ of $\la$ is considered multiplicatively closed if it satisfies two conditions: $\top\in S$, and for any $x\in S$ and $y\in S$, implies $x\pr y\in S$.
	\end{lemma}

	\begin{proof}
		Suppose that  $P$ is a prime ideal in $\la$ and $x,$ $y\in \la\neg P$. Then by Lemma \ref{lpsp}(\ref{kpp}), $x\pr y\notin P$, and hence $x\pr y\in \la\neg P.$ Conversely, suppose $x,$ $y\notin P$. This implies that $x,$ $y\in \la\neg P$. Since $\la\neg P$ is multiplicatively closed, $x\pr y\in \la\neg P$, and hence $x\pr y\notin P$, proving that $P$ is a prime ideal.
	\end{proof}

	\begin{lemma}\label{mxkp}
		Let $S$ be a multiplicatively closed subset in a quantale $\la$. Suppose $P\in \id$ such that it is maximal with respect to the property: $P\cap S=\emptyset$. Then $P\in \spc$.
	\end{lemma}

	\begin{proof}
		Suppose $x\notin P$ and $y\notin P$ for some $x,$ $y\in \la$.  By the property on $P$, this implies that the ideals $\langle x, P\rangle$ and $\langle y, P\rangle$ both intersect with $S$. Hence, there exist $l,$ $l'\in \la$, $p,$ $p'\in P$, and $s,$ $s'\in S$ such that $(x\pr l)\vee (p\pr s)\in S$ and $(y\pr l')\vee (p'\pr s')\in S$. Since $S$ is a multiplicatively closed set, we must have $((x\pr l)\vee (p\pr s))\pr ((y\pr l')\vee (p'\pr s'))\in S$. On the other hand,
		\begin{align*}
			((x\pr l)\vee (p\pr s))((y\pr l')\vee (p'\pr s'))&\\
			=(x\pr y\pr l\pr l')&\vee (p'\pr (x\pr l\pr s'))\vee (p\pr (y\pr l'\pr s))\vee (p\pr p'\pr (s\pr s'))\in \langle x\pr y, P\rangle,
		\end{align*}
		which implies $x\pr y\notin P$. Hence $P$ is a prime ideal.
	\end{proof}

	\begin{lemma}\label{rkt}
		Let $\la$ be a quantale. For every $I\in \id$, $\rd(I)$ is equal to the set \[T=\{l\in \la\mid \mathrm{every\; multiplicatively\; closed\; subset\; containing}\;l\;\mathrm{intersects}\;I\}.\]
	\end{lemma}

	\begin{proof}
		Suppose that  $r\in T$ and $P\in \spc$ such that $I\preccurlyeq  P$. Then by Lemma \ref{rpms}, $\la\neg P$ is a multiplicatively closed subset of $\la$ and $l\notin  \la\neg P.$ Hence $l\in P$. Conversely, let $l\notin T$. This implies that there exists a multiplicatively closed subset $S$ of $\la$ such that $l\in S$ and $S\wedge  I=\emptyset.$ By Zorn's lemma, there exists an ideal $P$ containing $I$ and maximal with respect to the property that $P\wedge  S=\emptyset$. By Lemma \ref{mxkp},  $P$ is a prime ideal with $l\notin P$.
	\end{proof}

	\begin{lemma}\label{mms}
		Suppose $I$ is a semiprime ideal in a quantale $\la$ and suppose $x\in \la\neg I$. Then there exists a multiplicatively closed subset $S$ of $\la$ such that $x\in S\preccurlyeq  \la\neg I.$
	\end{lemma}

	\begin{proof}
		Define the elements of $S=\{x_1, x_2, \ldots, x_n, \ldots\}$ inductively as follows:
		$x_1:= x;$
		$x_2:=x_1\pr x_1;$
		$\ldots$;
		$x_n:=x_{n-1}\pr x_{n-1};$ $\ldots$. Obviously, $x\in S$ and it is also easy to see that $x_i,$ $x_j\in S$ implies that $x_i\pr x_j\in S$.
	\end{proof}

	\begin{theorem}\label{spkr}
		For any ideal $I$ in a quantale $\la$, the following are equivalent.	\begin{enumerate}
			
			\item\label{iksp} $I$ is semiprime.
			
			\item\label{iinp} $I$ is an intersection of prime ideals in  $\la $.
			
			\item\label{iradi} $I$ is a radical ideal.
		\end{enumerate}
	\end{theorem}

	\begin{proof}
		From Proposition \ref{edr}, it follows that (3)$\Rightarrow$(2). Since the intersection of semiprime ideals is a semiprime ideal, (2)$\Rightarrow$(1) follows. What remains is to show that (1)$\Rightarrow$(3) and for that, it is sufficient to show  $\rd(I)\preccurlyeq I.$ Suppose that  $x\notin I$. Then $x\in \la\neg I$ and by Lemma \ref{mms}, there exists a multiplicatively closed subset $S$ of $\la$ such that $x\in S\preccurlyeq  \la\neg I.$ But $S\wedge  I=\emptyset$ and hence, by Lemma \ref{rkt}, $x\notin \rd(I).$
	\end{proof}

	\begin{corollary}
		If $I\in \id$, then $\rd(I)$ is the smallest semiprime ideal in $\la $ containing $I$.
	\end{corollary}
	
	Next we wish to consider saturation of a subset of a quantale and relate it to prime ideals.

	\begin{definition}
		Let $\la$ be a quantale and $S$ be a multiplicatively closed subset of $\la$. We say $S$ is \emph{saturated} if for $x$, $y\in \la$ and $x\pr y\in S$ implies that $x,$ $y\in S$. The \emph{saturation} of $S$ is defined as
		\[ \overline{S}=\{ x\in \la \mid \;\text{there exists}\;y\in \la\;\text{such that}\; x\pr y\in S\}.\]	
	\end{definition}
	
	The following proposition shows that $\overline{S}$ behaves as a closure operation. Moreover, it is related to prime ideals as expected and extends Exercise 7 (from Chapter 3) of \cite{AM69}.

	\begin{proposition}
		Let $\la$  be a quantale and $S$ be a multiplicatively closed subset of $\la$. Then the following holds.
		\begin{enumerate}
			\item $\overline{S}$ is the smallest saturated multiplicatively closed subset in $\la$ containing $S$.
			
			\item $S$ is saturated if and only if $\la\neg S$ is a join of prime ideals.
		\end{enumerate}
	\end{proposition}

	\begin{proof}
		(1) If $s\in S$, then, $\top\pr s=s\in \overline{S},$ and hence $S\preccurlyeq  \overline{S}.$ Suppose $x,$ $y\in \overline{S}$. Then there exist $y,$ $y'\in \la$ such that $x\pr x'\in S$ and $y\pr y'\in S$. Since $S$ is a multiplicatively closed set, \[(x\pr x')\pr (y\pr y')=(x\pr y)\pr (x'\pr y')\in S,\] which implies that $x\pr y\in \overline{S}.$ Therefore, $\overline{S}$ is a multiplicatively closed subset of $\la$. To show $\overline{S}$ is saturated, let $x\pr y\in \overline{S}$ for some $x,$ $y\in \la$. Then there exists $z\in \la$ such that $x\pr (y\pr z)=(x\pr y)\pr z\in S$. This implies $x\in \overline{S}.$ Similarly, $y\in \overline{S}.$ Finally, let $T$ be a saturated multiplicatively closed subset of $\la$ containing $S$. Suppose $s\in \overline{S}.$ Then there exists $y\in \la$ such that $x\pr y\in S\preccurlyeq  T$. Since $T$ is saturated, $x\in T$.
		
		(2) Suppose $\la\neg S=\bigvee_{ \lambda \in \Lambda}\{P_{\lambda}\mid P_{\lambda} \in \spc\}.$ Let $x\pr y\in S$ for some $x,$ $y\in \la$ and if possible, let $x\notin S$. This implies $x\preccurlyeq p_{\lambda}$ for some $p_{\lambda}\in P_{\lambda}$. Since $P_{\lambda}$ is an ideal and $x\pr y\preccurlyeq x\wedge y\preccurlyeq x \preccurlyeq p_{\lambda} \in P_{\lambda},$ we must have $x\pr y \in P_{\lambda}.$ This implies $x\pr y\notin S$, a contradiction. This proves that $S$ is saturated.
		Conversely, let $S$ be a saturated subset of $\la$. Let $a\in \la\neg S.$ It suffices to show that $a\in P_{\lambda}$ for some $P_{\lambda}\in \spc$ such that $P_{\lambda}\wedge S=\emptyset.$ Suppose \[\Omega=\{ I\in \id\mid a\in I\;\text{and}\; I\wedge S=\emptyset\}.\] Observe that $\langle a \rangle \wedge S = \emptyset$. Indeed, $x_i\pr a\preccurlyeq \bigvee_{i=1}^n x_i\pr a$ for $x_i\in \la,$  implies $x_i\pr a \in S$ and since $S$ is saturated; this implies $a \in S$, contradiction. Hence $\langle a \rangle\in \Omega$, and thus $\Omega\neq \emptyset.$ By Zorn's lemma,
		$\Omega$ has a maximal element, say $P$. We claim that $P$
		is a prime ideal, and  proof of this is already described in  Proposition \ref{edr}.
	\end{proof}
	
	\subsection{Primary ideals and primary decompositions}\label{ppd}

	The decomposition of an ideal into primary ideals holds a significant place in the ideal theory of rings and is considered a traditional cornerstone. The objective of this section is to establish several classical uniqueness theorems for ideals for quantales and hence extend the corresponding results of \cite{AM69}.

	\begin{definition}
		A proper ideal $P$ in a quantale $\la$ is called \emph{primary} if for $x,$ $y\in \la$ and $x\pr y\in I$ imply that $x\in I$ or $y^n\in I$ for some  $n\in \mathds{N}$. If $P'$ is a primary ideal in $\la$ and $\rd(P')=P$, then we say $P'$ is \emph{$P$-primary}.	
	\end{definition}

	\begin{proposition}
		Let $\la$ be a quantale. Then the following holds.
		\begin{enumerate}
			
			\item Every prime ideal in $\la$ is primary.
			
			\item Let $P$ be a primary ideal in $\la$. Then $\rd(P)$ is the smallest prime ideal in $\la$ containing $P$. 
		\end{enumerate}
	\end{proposition}

	\begin{proof}
		(1) Suppose $P$ is a prime ideal in $\la$ and $x\pr y\in P$ for some $x, $ $y\in \la$. Then $x\in P$ or $y=y^1\in P,$
		
		(2) By Proposition \ref{edr}, it suffices to show that $\rd(P)$ is a prime ideal. Suppose $x\pr y\in \rd(P)$ for some $x, $ $y\in \la$. Then $(x\pr y)^n\in P$ for some $n\in \mathds{N}$. This implies that $x^n\in P$ or $y^{mn}\in P$ for some $m\in \mathds{N}.$ Hence, $x\in \rd(P)$ or $y\in \rd(P).$
	\end{proof}

	When considering a finite intersection of primary ideals in a quantale, with the condition that all the ideals involved are $P$-primary for a given prime ideal $P$, the resulting intersection is indeed primary. To establish this claim, we first require a preliminary result.

	\begin{lemma}\label{plpd}
		If $I_1,\ldots, I_n$ are ideals in a quantale $\la$ and $I=\bigwedge_{i=1}^n I_i,$ then $\rd(I)=\bigwedge_{i=1}^n \rd(I_i).$
	\end{lemma}
	
	\begin{proof}
		Suppose $x\in \rd(I)$. Then $x^n\in I$ for some $n\in \mathds{N}$. This implies that $x^n\in I_i$ for all $i\in \{1, \ldots, n\},$ and hence $x\in \rd(I_i)$ for all $i\in \{1, \ldots, n\},$ \textit{i.e.}, $x\in \bigwedge_{i=1}^n \rd(I_i).$ Now, to obtain the other inclusion, let $z\in \bigwedge_{i=1}^n \rd(I_i).$ This implies that $z^{m_i}\in \rd(I_i),$  where $m_i\in \mathds{N},$ for all $i\in \{1, \ldots, n\}.$ Choose $m=\max_{1\preccurlyeq i \preccurlyeq n}\{m_i\}$. Then $z^m\in I_i$ for all $i\in \{1, \ldots, n\},$ and hence $z\in \rd(I).$
	\end{proof}

	\begin{theorem}\label{piqp}
		If $P$ is a prime ideal in a quantale $\la$ and if $P_1,\ldots, P_n$ are $P$-primary ideals in $\la$, then $P'=\bigwedge_{i=1}^n P_i$ is also a $P$-primary ideal in $\la$.
	\end{theorem}

	\begin{proof}
		By Lemma \ref{radpr}(\ref{ijicj}) and Lemma \ref{plpd}, we have 
		\[\rd(P')=\rd\left(\bigwedge_{i=1}^n P_i \right)=\bigwedge_{i=1}^n\rd(P_i)=\bigwedge_{i=1}^nP=P.\]
		What remains for us to prove is that $P'$ is a primary ideal. Let $x\pr y\in P'$. Then $x\pr y\in P_i$ for all $1\preccurlyeq i\preccurlyeq n$. If $x\notin P_j$ for some $j$, then $y^m\in P_j$ for some $m\in \mathds{N}$, because $P_j$ is a primary ideal. It follows that $y\in \rd(P_j)=P$. Since $\rd(P')=P$, there exists $n\in \mathds{N}$ such that $y^n\in P'$.
	\end{proof}

	\begin{proposition}\label{pqx}
		Suppose $P$ is a prime ideal in a quantale $\la$, $P'$ is a $P$-primary ideal, and $x$ is an element in $\la$. Then the following hold.
		\begin{enumerate}
			\item\label{fiq} If $x\in P'$ then $(P':x)=\la$.
			
			\item\label{siq} If $x\notin P'$, then $(P':x)$ is a $P$-primary ideal.
			
			\item If $x\notin P$, then $(P':x)=P'$.
		\end{enumerate}
	\end{proposition}

	\begin{proof}
		(1) If $x\in P'$, then $x\pr 1=x\in P'$, and hence $1\in (P':x),$ implying that $(P':x)=\la$.
		
		(2) Let $x\notin P'$. If $y\in (P':x)$, then $x\pr y\in P'$. Since $P'$ is a primary ideal and $x\notin P'$, there exists $n\in \mathds{N}$ such that $y^n\in P'$. Hence $y\in \rd(P')=P$. Therefore, we have $P'\preccurlyeq (P':x)\preccurlyeq P$, implying that 
		\[P=\rd(P')\preccurlyeq \rd(P':x)\preccurlyeq \rd(P)=P.\]
		Hence $\rd(P':x)=P$. It remains for us to show that $(P':x)$ is a primary ideal. Since $x\notin P'$, we have $\top\notin (P':x)$, so $(P':x)\neq \la$, \textit{i.e.} is a proper ideal in $\la$. Let $a\pr b\in (P':x)$. If $b^m\notin P'$ for all $m\in \mathds{N}$, then $b\notin \rd(P':x)=P$. However, $a\pr b\pr x\in P'$ implies that $a\pr x\in P'$ or $b^k\in P'$ for some $k\in \mathds{N}$, because $P'$ is a primary ideal. If $b^k\in P'$, then $b\in \rd(P')=P,$ a contradiction. Therefore, $a\pr x\in P',$ implying that $a\in (P':x)$whence $(P':x)$ is a primary ideal. 
		
		(3) Let $x\notin P$. If $y\notin P'$ and $x\pr y\in P'$, then $x^n\in P'$ for some $n\in \mathds{N}$, because $P'$ is a primary ideal. Therefore $x\in \rd(P')=P$, a contradiction. Hence $x\pr y\notin P'$. This implies that $y\notin (P':x)$. By Proposition \ref{bpi}(\ref{iicj}), $P'\preccurlyeq (P:x),$ and this we have $P'=(P':x).$	
	\end{proof}

	\begin{definition}\label{pdd}
		A \emph{primary decomposition} of an ideal $I$ in a quantale $\la$ is an expression:
		\begin{equation}\label{prd}
			I=\bigwedge_{i=1}^n P_i,
		\end{equation}
		where $P_i$ are primary ideals in $\la$. The expression (\ref{prd}) is said to be \emph{minimal} if
		\begin{enumerate}
			\item\label{radd} $\rd(P_1), \ldots, \rd(P_n)$ are distinct.
			
			\item\label{mpj} $\bigwedge_{j\neq i}P_j\nsubseteq P_i,$ for all $i.$ 
		\end{enumerate}
		If an ideal $I$ in $\la$ has a primary decomposition, then we say that $I$ is \emph{decomposable}.	
	\end{definition}

	\begin{proposition}
		A primary decomposition may be replaced by a minimal primary decomposition.
	\end{proposition}

	\begin{proof}
		Consider a primary decomposition $I=\bigwedge_{i=1}^n P_i,$ where $P_i$ are primary ideal in $\la$. If \[\rd(P_{i_1})=\cdots =\rd(P_{i_k})=P,\] then by Theorem \ref{piqp}, $P'=\bigwedge_{j=1}^k P_{i_j}$ is a $P$-primary ideal. Therefore, we can replace $P_{i_1}, \ldots, P_{i_k}$ by $P'$, and continue the process to guarantee that condition (\ref{radd}) of Definition \ref{pdd} holds.	If condition (\ref{mpj}) of Definition \ref{pdd} does not hold, then we can eliminate ideals until it does hold, without changing the overall intersection.
	\end{proof}

	While primary decompositions, or minimal primary decompositions, may not be unique, they possess certain uniqueness properties, which are outlined in the following theorem.

	\begin{theorem}
		Let $I$ be a decomposable ideal in a quantale and $I=\bigwedge_{i=1}P_i$, a minimal primary decomposition. Let $P_i'=\rd(P_i),$ for $i=1, \ldots, n$. Then the set $\{P_1', \ldots, P_n'\}$ is composed of the prime ideals $P$ of $\la$ such that $P=\rd((I:x)),$ for some $x\in \la$.
	\end{theorem}

	\begin{proof}
		Suppose $x\in \la$. By Proposition \ref{bpi}(\ref{inqi}) and Lemma \ref{plpd}, we have
		\[\rd((I:x))=\rd\left(\left( \bigwedge_{i=1}^n P_i:x \right)\right)=\rd\left(  \bigwedge_{i=1}^n (P_i:x)\right)=\bigwedge_{i=1}^n\rd((P_i:x)). \]	
		From (\ref{fiq}) and (\ref{siq}) of Proposition \ref{pqx}, we obtain		
		\[\rd((I:x))=\bigwedge_{i=1}^n\rd((P_i:x))=\bigwedge_{i, x\notin P_i}P_i'.\]
		If the intersection of a finite set of ideals is a prime ideal, then the intersection is equal to one of
		the ideals; thus, if $\rd(I:x)$ is a prime ideal, then
		\[\rd((I:x))\preccurlyeq \{P_i'\mid x\notin P_i\}\preccurlyeq \{P_1', \ldots, P_n'\}.\]
		For the converse, let $i\in \{1, \ldots, n\}.$ Because the primary decomposition is
		minimal, for each $i$, there exists $x_i\in \left(\bigwedge_{j\neq i} P_j \right)\neg P_i.$ If $y\in (P_i:x_i),$ then $y\pr x_i\in P_i$. Therefore,
		\[y\pr x_i\in P_i\wedge \left( \bigwedge_{j\neq i}P_j \right)=I, \]
		which implies that $y\in (I:x_i).$ Hence
		\[(P_i:x_i)\preccurlyeq (I:x_i)\preccurlyeq (P_i':x_i), \]
		where the second inclusion follows from the fact that $I\preccurlyeq  P_i$. Therefore, $(P_i:x_i)=(I:x_i).$ By Proposition \ref{pqx}(\ref{siq}), we have
		\[\rd((I:x_i))=\rd((P_i:x_i))=P_i'.\]
		From this it follows that $P_i'$ form the set of those ideals $\rd((I:x))$ that are prime ideals in $\la$.
	\end{proof}

	Suppose $I$ is a decomposable ideal and $I=\bigwedge_{i=1}^nP_i$, a minimal primary decomposition. Let us denote $\rd(P_i)=P_i'$. We say that the prime ideals $P_1',\ldots, P_n'$ \emph{belong to} $I$. The
	minimal elements of the set $S = \{P_1',\ldots, P_n'\}$ with respect to inclusion are said to be \emph{isolated} prime ideals belonging to $I$,
	and the others are called \emph{embedded} prime ideals. Define \[I^{\uparrow_{\spc}}=\{P\in \spc\mid P\supseteq I\}.\] We shall show that the minimal elements of the set S are the minimal elements of
	the set $I^{\uparrow_{\spc}}$. To obtain that, we need the following result.

	\begin{proposition}\label{idqa}
		Let $I$ be a decomposable ideal in a quantale $\la$, $I = \bigwedge_{i=1}^n P_i$ a minimal primary
		decomposition, with $S$ the set of prime ideals belonging to $I$ . If $P\in I^{\uparrow_{\spc}}$, then $P$ contains an
		isolated prime ideal $P_j'$.
	\end{proposition}

	\begin{proof}
		Set $P_i'=\rd(P_i),$ for $i\in \{1, \ldots, n\}.$ Applying Lemma \ref{plpd}, we obtain
		\[\bigwedge_{i=1}^nP_i'=\bigwedge_{i=1}^n\rd(P_i)=\rd(I)\preccurlyeq \rd(P)=P .\]
		However, if a prime ideal contains an intersection of ideals, then
		at least one of the ideals in the intersection is contained in the prime ideal; therefore $P_j'\preccurlyeq P$, for
		some $j$. The result now follows.
	\end{proof}

	\begin{corollary}
		The isolated prime ideals are the minimal elements in $I^{\uparrow_{\spc}}.$
	\end{corollary}

	\begin{proof}
		In Proposition \ref{idqa}, we have observed that  $S \preccurlyeq I^{\uparrow_{\spc}}$, so a fortiori, if $P_j'$ is an isolated prime in $S$, then $P_j\in I^{\uparrow_{\spc}}$.
		Suppose now that $P\in I^{\uparrow_{\spc}}$, with $P\preccurlyeq   P_j'$. From Proposition \ref{idqa}, it follows that there exists an isolated prime ideal
		$P_k' \preccurlyeq P$, so $P_k' \preccurlyeq P_j'$. Since $P_j'$ is isolated, $P_k' = P_j'$, hence $P = P_j'$, and it follows that $P_j'$ is
		minimal in $I^{\uparrow_{\spc}}$.			
	\end{proof}

	Although the primary ideals in different minimal decompositions of an ideal are not necessarily
	the same, the primary ideals whose radicals are isolated are the same. We aim now to establish
	this.

	\begin{lemma}
		Let $I$ be a decomposable ideal in a quantale $\la$ and $P_j$ an ideal in a minimal decomposition of $I$ such that $\rd(P_j)$ is an isolated prime ideal. Then $P_j$ is composed of the elements $a\in \la$
		for which there exists $b\notin \rd(P_j)$ with $a\pr b\in I$.
	\end{lemma}

	\begin{proof}
		Suppose $I = \bigwedge_{i=1}^n P_i$ is a minimal decomposition, with $\rd(P_j)$ isolated. We claim that $P_i\nsubseteq \rd(P_j)$, for $i\neq j$. If possible, let $P_i\preccurlyeq  \rd(P_j)$  and let $x\in \rd(P_i)$; then $x^n\in P_i$, for some $n\in \mathds{N},$
		which implies that there exists $m\in \mathds{N}$ such that $x^{nm}\in P_j$,  because $P_i\preccurlyeq  \rd(P_j)$,
		hence $x \in \rd(P_j )$. It follows that $\rd(P_i) \preccurlyeq \rd(P_j)$, which is impossible, because $\rd(P_j)$ is isolated.
		Thus $P_i \nleqslant \rd(P_j)$, as claimed.		
		
		Now
		let $a\in P_j$. From what we have just seen, for $i\neq j$, there exists $b_i\in P_i \neg \rd(P_j).$ Set
		$b = \prod_{i\neq j} b_i$. As $\rd(P_j)$ is a prime ideal, we have $b\notin \rd(P_j)$. However $a\pr b \in P_i$, for $i\neq j$, because $b_i\in P_i$. Also, given that $a\in P_j$, we also have, $a\pr b \in P_j$ and so \[a\pr b \in \bigwedge_{i=1}^nP_i = I.\] Therefore, $a$ is
		an element in $\la$ for which there exists $b \notin \rd(P_j)$ with $a\pr b \in I$.
		We now consider the converse. Suppose that $a \in \la$ and $b\notin \rd(P_j)$ are such that $a\pr b \in I$. Since
		$I\preccurlyeq  P_j$, we have $a\pr b \in P_j$. As $P_j$ is primary and $b^n\notin 
		P_j$, for any $n\in \mathds{N}$, we must have $a \in P_j$.
		This concludes the proof. 
	\end{proof}

	\begin{corollary}
		If $I$ is a decomposable ideal in a quantale $\la$, then the primary ideals in a minimal
		decomposition of $I$ corresponding to isolated prime ideals are uniquely determined.
	\end{corollary}
	
	\subsection{Irreducible and strongly irreducible ideals}\label{isii}
	
	The concept of strongly irreducible ideals, originally known as primitive ideals, was introduced in \cite{Fuc49} for commutative rings. In \cite[p.\,301, Exercise 34]{Bou72}, these ideals are referred to as quasi-prime. The term ``strongly irreducible'' was initially employed for noncommutative rings in \cite{Bla53}. Strongly irreducible ideals (of rings) are significant from both algebraic and topological aspects. For the detail, we refer our reader to \cite{Azi08, Bla53, DG23, FGS23}.

	\begin{definition}\label{doi}
		Suppose $\la $ is a quantale.
		\begin{enumerate}
			\item An ideal $I$ of $\la$ is called \emph{irreducible} if  $J\wedge  J'= I$ implies that either $J= I$ or $J'= I$ for all $J,$ $J'\in \id$.
			
			\item\label{dsi} An ideal $I$ of $\la$ is called \emph{strongly irreducible} if  $J\wedge  J'\preccurlyeq I$ implies that either $J\preccurlyeq  I$ or $J'\preccurlyeq I$ for all $J,$ $ J'\in \id$.
		\end{enumerate}
		We shall use the notation $\mathrm{Irr}(\la)$ to represent the set of irreducible ideals in $\la$ and $\mathrm{Irr}^+(\la)$ to denote the set of strongly irreducible ideals in $\la$. 
	\end{definition}
	
	The following result gives an equivalent definition of a strongly irreducible ideal in terms of elements of a quantale, and it extends Proposition 7.33 of \cite{G99} and Theorem 20 of \cite{Bla53}.

	\begin{proposition}
		\label{abi} 
		An ideal $I$ in a quantale $\la $ is strongly irreducible if and only if	for all $a$, $b\in \la $ satisfying the condition\,$:$ $\langle a\rangle  \wedge  \langle b\rangle \preccurlyeq  I$ implies either $a\in I$ or $b\in I$.
	\end{proposition}

	\begin{proof}
		Suppose that $I$ is a strongly irreducible ideal in $\la$ and $\langle a\rangle  \wedge  \langle b\rangle \preccurlyeq  I$ for some $a,$ $b\in \la.$ 
		Since $I$ is strongly irreducible, we must have either $a\in \langle a\rangle\preccurlyeq  I$ or  $b\in \langle b\rangle\preccurlyeq  I$. For the converse, suppose that $I$ is an ideal in $\la$ that it is not strongly irreducible. Then there are ideals $J$ and $K$ of $\la$ satisfy $J\wedge  K\preccurlyeq  I,$ but $J\nleqslant I$ and $K\nleqslant I.$ This implies there exists $j\in J$ and $k\in K$ such that $\langle j\rangle \wedge  \langle k\rangle \preccurlyeq  I$, but neither $j\notin I$ nor $k\in I$, a contradiction.
	\end{proof}

	In Theorem \ref{spkr}, it was noted that a radical ideal could be presented as the intersection of prime ideals that contain it. Now, we aim to illustrate in the next proposition that any proper ideal can be expressed in a comparable fashion, albeit utilizing irreducible ideals instead. This result extends Proposition 7.35 of \cite{G99}. However, prior to advancing further, we must establish a lemma.

	\begin{lemma}\label{lir}
		Consider a quantale $\la$. Assume that $\bot \neq x \in \la$ and $I$ is a proper ideal in $\la$ such that $x \notin I$. In this case, there exists an irreducible ideal $J$ of $\la$ satisfying the conditions $I \preccurlyeq J$ and $x \notin J.$
	\end{lemma}

	\begin{proof}
		Let
		$\{J_{\lambda}\}_{\lambda \in \Lambda}$ be a chain of ideals in $\la$ such that $I\preccurlyeq J_{\lambda}$ and $x\notin  J_{\lambda}$ for all $\lambda \in \Lambda$. Then
		$I\preccurlyeq  \bigvee_{\lambda \in \Lambda} J_{\lambda}
		$ and $x\notin  J_{\lambda}$, for all $\lambda \in \Lambda$.
		By Zorn's lemma, there exists a maximal element $J$ of this chain. Suppose that  $J=J_1\wedge  J_2.$ By the maximality condition of $J$, we must have $x\in J_1$ and $x\in J_2,$ and hence $x\in J_1\wedge  J_2=J,$ a contradiction. Therefore, $J$ is the required irreducible ideal.
	\end{proof}
	
	\begin{proposition}
		If $I$ is a proper ideal in a quantale $\la$, then $I=\bigwedge_{I\preccurlyeq  J} \{J\mid J \in \mathrm{Irr}(\la) \}$.
	\end{proposition}

	\begin{proof}
		By Lemma \ref{lir}, there exists an irreducible ideal $J$ of  $\la$ such that $I\preccurlyeq  J$. Let \[J'=\bigwedge_{I\preccurlyeq  J} \{J\mid J\in \mathrm{Irr}(\la) \}.\]
		Then $I\preccurlyeq  J'$. We claim that $J'=I.$   If $J'\neq I$, then there exists an $x\in J'\neg I$, and by Lemma \ref{lir}, there exists an irreducible ideal $J''$ such that $x\notin J''$ and $I\preccurlyeq J''$, a contradiction.
	\end{proof}

	In the context of a Noetherian ring, it is widely recognized that radicals can be expressed as a finite intersection of prime ideals. Expanding upon this idea, the following proposition extends this result to include any ideal by utilizing irreducible ideals. This result also extends Proposition 7.3 of \cite{N18}. We say a quantale \emph{Noetherian} if every ascending chain of ideals eventually becomes stationary.

	\begin{proposition}
		If $\la$ is a  Noetherian quantale, then every ideal in $\la $ can be represented as the intersection of a finite number of irreducible ideals in $\la$.
	\end{proposition}

	\begin{proof}
		Suppose that 
		\[\mathcal{F}=\left\{J\in  \rd(\la) \mid J\neq \bigwedge_{j=1}^n I_j,\;I_j \in \mathrm{Irr}(\la)\right\}.\]  It is sufficient to show that $\mathcal{F}=\emptyset$.
		Since $\la$ is Noetherian, $\mathcal{F}$ has a maximal element, say $I$. Since $I \in \mathcal{F}$, it is not a finite
		intersection of irreducible ideals in $\la$. This implies that $I$ is not  irreducible. Hence, there are ideals $J$ and $K$ such that  $I\preccurlyeq J$, $I\preccurlyeq K$, $I\neq J$, $I\neq K$, and $I = J \wedge  K.$ Since $I$ is
		a maximal element of $\mathcal{F}$, we must have $J,$ $K \notin \mathcal{F}.$ Therefore, $J$ and $K$ are a finite intersection of
		irreducible ideals in $\la$, which followingly implies that $I $ is also a finite intersection of irreducible
		ideals in $\la$, a contradiction.
	\end{proof}

	The following proposition establishes the relationships between prime, semiprime, and strongly irreducible ideals, and it extends Proposition 7.36 of \cite{G99} and Theorem 2.1(i) of \cite{Azi08}.

	\begin{proposition}
		
		\label{prira} 
		If $P$ is a strongly irreducible ideal in  $\la$, then $P$ is prime if and only if $P$ is radical.
	\end{proposition}

	\begin{proof}
		Notice that if $P$ is prime, then obviously $P$ is a radical ideal. For the converse,  suppose that  $P$ is a radical ideal and $I\pr J\preccurlyeq  P$, for some $I,$ $J\in \id .$ Then \[I\wedge  J\preccurlyeq  \rd(I\wedge  J)=\rd(I\pr J)\preccurlyeq \rd(P)=P,\] where the first inclusion and the first equality respectively follow from Lemma \ref{radpr}(\ref{irad}) and Lemma \ref{radpr}(\ref{ijicj}). Moreover, the second inclusion and the last equality respectively follow from Lemma \ref{radpr}(\ref{ijrirj}) and the fact that $P$ is a prime ideal. Since $P$ is strongly irreducible, either $I\preccurlyeq  P$ or $J\preccurlyeq  P.$
	\end{proof}

	In the context of commutative rings, it has been established (see \cite[Theorem 2.1(ii)]{Azi08}) that every proper ideal is   contained in a minimal strongly irreducible ideal. Remarkably, the same principle applies to strongly irreducible ideals in quantales.

	\begin{proposition}
		Every proper ideal in a quantale is contained in a minimal strongly irreducible ideal.
	\end{proposition}

	\begin{proof}
		Let $I$ be a proper ideal in a quantale $\la$ and let \[\mathcal{E}=\{J\mid I\preccurlyeq  J,\;J\in \mathrm{Irr}^+(\la)\}.\]
		Note that every maximal ideal in $\la $ is prime, and by Lemma \ref{bpi}(\ref{mip}), every prime ideal is strongly irreducible. Also, observe that by Theorem \ref{emi}, every proper ideal is contained in a maximal ideal. Therefore, the set $\mathcal{E}$ is nonempty. By Zorn's lemma, $\mathcal{E}$ has a minimal element, which is our desired minimal strongly irreducible ideal.
	\end{proof}

	The following result demonstrates the scenario in which all ideals in a quantale are strongly irreducible, and its proof is evident. This result extends Lemma 3.5 of \cite{Azi08}.

	\begin{proposition}
		Every ideal in a quantale $\la$ is strongly irreducible if and only if  $\id $ is totally ordered. 
	\end{proposition}

	We bring this subsection to a close with a theorem pertaining to arithmetic quantales, wherein the notions of irreducibility and strongly irreducibility align. The proof of this theorem is identical to that of \cite[Theorem 3 and Theorem 7]{Is{e}56}

	\begin{theorem}
		In an arithmetic quantale $\la$, an ideal in $\la$ is
		irreducible if and only if it is strongly irreducible. Conversely, 
		if an irreducible ideal in a quantale $\la$ is strongly
		irreducible, then $\la $ is arithmetic.
	\end{theorem}

	\begin{corollary}
		In an arithmetic quantale, any ideal is the intersection of all strongly irreducible ideals containing it.
	\end{corollary}
	
	\begin{cremark}
		\label{fuw}
		
		Building upon the framework of the ``ideal theory of quantales'' that we have begun to develop in this paper, we shall now provide a brief outline of our future work. In \cite{DG23'}, our focus  is on examining the properties of different ideals in relation to the quotient of a quantale. We shall also introduce the concept of localization for a quantale. Additionally, by considering modules over quantales, we shall extend some of the significant results in commutative algebra. 
		
		On the other hand, in \cite{DG23''}, our objective  is to establish lower topologies on distinguished classes of ideals in a quantale, and this  extends the work done in \cite{DG23} and \cite{FGS23} for rings. We shall investigate the topological properties of ``quantale spaces,'' including quasi-compactness, separation properties, connectedness, continuity, and spectral spaces (as defined in \cite{Hoc69}). Specifically, we shall delve into the detailed examination of quantale spaces associated with prime ideals, minimal prime ideals, maximal ideals, and strongly irreducible ideals in a quantale. These investigations  respectively extend the work carried out in \cite{Gro60}, \cite{HJ65}, \cite{Sto37}, and \cite{Azi08} for rings. Furthermore, we shall demonstrate that the quantale space of proper ideals is spectral, thereby generalizing the corresponding result presented in \cite{Gos23} for rings.
		
		Since the definition of quantale we have considered here essentially gives semirings, it is worth it to see how many exclusive types of ideals ($k$-ideals, strong ideals, semisubtractive ideals) of semirings have their abstraction in our setup.  For example, $k$-ideals studied in \cite{GD24} will become all ideals in quantales.
		
		While our focus in this paper was on a commutative quantale with identity $1$, it is worth noting that the developed theory can be further extended to a complete lattice $(\mathsf{L}, \preceq,  \bot, \top)$ equipped with an operation denoted as $\odot$ that retains Axioms (1), (2), and (4) from Definition \ref{mld}, as well as the condition: \[x\odot (y\vee z)=(x\odot y) \vee (x\odot z),\] for all $x,$ $y,$ and $z$ in $\mathsf{L}.$
		
		The attentive reader may have observed that we have thus far assumed commutativity in our quantales. However, by relaxing the commutativity axiom, we naturally arrive at the concept of primitive ideals (as described in \cite{J56}) in quantales. In the forthcoming paper \cite{Gos23'}, our objective  is to explore Jacobson's structure theory for quantales. Specifically, we shall formulate and prove the Jacobson-Cheveley Density Theorem for quantales. In line with \cite{J45}, we shall also examine the Jacobson topology on the primitive ideals in a quantale.
	\end{cremark}
	


\begin{thebibliography}{100}
		
		\bibitem{And74} D. D. Anderson, \textit{Multiplicative lattices}, PhD dissertation, University of Chicago, 1974. MRNUMBER = {2611710}
		
		\bibitem{Azi08}  A. Azizi, Strongly irreducible ideals, \emph{J. Aust. Math. Soc.}, 84 (2008), 145--154. DOI = {10.1017/S1446788708000062}
		
		\bibitem{AJ84}
		\bysame and E. W. Johnson, Ideal theory in commutative semigroups, \textit{Semigroup Forum}, 
		30 (1984), 
		127--158. DOI = {10.1007/BF02573445}
		
		\bibitem{AJA95}
		\bysame, C. Jayaram, and F. Alarcon, Some results on abstract commutative
		ideal theory, \textit{Period. Math. Hungar.}, 
		30  (1995), 1--26. DOI = {10.1007/BF01876923}
		
		\bibitem{AK13}  A. Altman and S. Kleiman, A Term of Commutative Algebra, \textit{Worldwide Centre of Mathematics, LLC}, 2013. ISBN 978-0-9885572-1-5
		
		\bibitem{AM69}
		M. F. Atiyah and I. G. Macdonald. \textit{Introduction to commutative algebra}, Addison-Wesley, 1969. MRNUMBER = 242802
		
		\bibitem{Bla53}  R. L.
		Blair, Ideal lattices and the structure of rings, \emph{Trans. the Amer. Math. Soc.}, 75 (1953),  136--153.   DOI = {10.2307/1990782}
		
		\bibitem{Bou72}  N. Bourbaki, \emph{Elements of mathematics:
			Commutative
			Algebra}, Addison-Wesley, Reading, MA, 1972. MRNUMBER = {360549}
		
		\bibitem{Bur75}
		R. G. Burton, Fractional elements in multiplicative lattices, \textit{Pacific J. Math.}, 56(1) (1975), 35--49. MRNUMBER = {374113}
		
		\bibitem{BAS16}  F. B. Bergamaschi, A. O. Andrade, and R. H. N. Santiago, Prime ideals over noncommutative quantales,  in 
		Uncertainty 
		modelling in knowledge engineering and decision making, 432--437 (2016), \textit{World Scientific Proceedings Series on Computer Engineering and Information Science}.
		
		\bibitem{CZ20}
		R. Coleman and L. Zwald, Minimal ideals and some applications, (2020) hal-03040754.
		
		\bibitem{Dil62}
		R. P. Dilworth,  Abstract commutative ideal theory, \textit{Pacific J. Math.}, 12 (1962), 481--498. MRNUMBER = {143781}
		
		\bibitem{Dil36}
		\bysame, Abstract residuation over lattices, \textit{Bull. Amer. Math.
			Soc.}, 44
		(1938), 262--268. DOI = {10.1090/S0002-9904-1938-06736-5}
		
		\bibitem{DG23}
		T. Dube and A. Goswami, Ideal spaces: An extension of structure spaces of rings, \textit{J. Algebra Appl.},  22(11), (2023), 2350245 (18 pages). DOI = {10.1142/S0219498823502456}
		
		\bibitem{DST19}  M. Dickmann, N. Schwartz, and M. Tressel,  \emph{Spectral spaces}, Cambridge Univ. Press, 2019. DOI = {10.1017/9781316543870}
		
		\bibitem{Fac23}
		A. Facchini, Algebraic structures from the point of view of complete multiplicative lattices, in \textit{Algebra and coding theory}, 113--131,
		\textit{Contemp. Math.}, 785, Amer. Math. Soc., 2023. DOI = {10.1090/conm/785/15780}
		
		\bibitem{FGS23}  C. A. Finocchiaro, A. Goswami, and D. Spirito, Distinguished classes of ideal spaces and their topological
		properties,
		\textit{Comm. Algebra}, 51(4), (2023), 1752--1760. DOI = {10.1080/00927872.2022.2141767}
		
		\bibitem{FFJ22}
		\bysame, C. A. Finocchiaro and G. Janelidze, Abstractly constructed prime spectra, \textit{Algebra Universalis}, 83(8) 
		(2022), 38 pp. DOI = {10.1007/s00012-021-00764-z}
		
		\bibitem{Fuc50}
		L. Fuchs, On primal ideals, \textit{Proc. Amer. Math. Soc.}, 1 (1950), 1--6. DOI = {10.2307/2032421}
		
		\bibitem{Fuc49}  \bysame, \"{U}ber die Ideale arithmetischer Ringe, \emph{Comment. Math. Helv.}, 23 (1949), 334--341. DOI = {10.1007/BF02565607}
		
		\bibitem{Fuc48}
		\bysame, A condition under which an irreducible ideal is primary, \textit{Quart. J. Math. Oxford},
		19 (1948), 235--237. DOI = {10.1093/qmath/os-19.1.235}
		
		\bibitem{FHO04}
		\bysame, W. Heinzer, and B. Olberding, Commutative ideal theory without finiteness conditions: primal ideals, 
		\textit{Trans. Amer. Math. Soc.}, 357(7) (2004), 2771--2798. DOI = {10.1090/S0002-9947-04-03583-4}
		
		\bibitem{G99} 
		J. S. Golan, \textit{Semirings and their applications}, Springer, 1999. DOI = {10.1007/978-94-015-9333-5}
		
		\bibitem{Gos23}  A. Goswami, Proper spaces are spectral, \textit{Appl. Gen. Topol.}, 24(1), (2023), 95--99. DOI = 10.4995/agt.2023.17800
		
		\bibitem{Gos24}
		\bysame, Lower spaces of multiplicative lattices, \textit{Boll. Unione Mat. Ital.} (online ready).
		
		\bibitem{DG23'}  \bysame, On ideals in quantales, II (in preparation).
		
		\bibitem{DG23''}  \bysame, Some topological aspects of ideals in quantales, \textit{Appl. Gen. Topol.}, 26(1), (2025),  483--499. DOI = {10.4995/agt.2025.22497}
		
		\bibitem{Gos23'}  \bysame, On primitive ideals in quantales (in preparation).
		
		\bibitem{GD24}
		\bysame and T. Dube, Some aspects of $k$-ideals of semirings, \textit{Rend. Circ. Mat. Palermo (2)}, 73, (2024),  3105--3117. DOI = {10.1007/s12215-024-01097-1}
		
		\bibitem{Gro60} 
		A. Grothendieck,   \emph{\'{E}l\'{e}ments de g\'{e}om\'{e}trie alg\'{e}brique I, Le langage des sch\'{e}mas}, Inst. Hautes \'{E}tudes Sci. Publ. 
		Math., No. 4., 1960. MRNUMBER = {217083}
		
		\bibitem{Get.al.03} G. Gierz, K. H. Hofmann, K. Keimel, J. D. Lawson,
		M. Mislove, and D. S. Scott,  \emph{Continuous lattices and domains},
		Cambridge Univ. Press, 2003. DOI = {10.1017/CBO9780511542725}
		
		\bibitem{Har73} D. Harris, Universal compact $T_{1}$ spaces,       
		\textit{General Topology and Appl.}, 3 (1973), 291--318. MRNUMBER = {331325}
		
		\bibitem{Hoc69} 
		M. Hochster,  Prime ideal structure in commutative rings, \textit{Trans. Am. Math. Soc.}, 142 (1969), 43--60. DOI = 10.2307/1995344
		
		\bibitem{HJ65}
		M. Henriksen and M. Jerison,  The space of minimal prime ideals of a commutative ring, 
		\textit{Trans. Amer. Math. Soc.}, 115 (1965), 110--130. DOI = {10.2307/1994260}
		
		\bibitem{Is{e}56}
		K. Is\'{e}ki, Ideal theory of semirings, \textit{Proc. Japan. Acad.}, 32 (1956), 554--559. MRNUMBER = {85244}
		
		\bibitem{J45} 
		N. Jacobson, 
		\textit{A topology for the set of primitive ideals in an arbitrary rings}, Proc.
		Nat. Acad. Sei. U.S.A., 31 
		(1945), 333--338. DOI = {10.1073/pnas.31.10.333}
		
		\bibitem{J56}  \bysame, \textit{Structure of rings}, Amer. Math. Soc. Colloquium Publications,
		vol. 37, Providence, (1956). MRNUMBER = {81264}
		
		\bibitem{Kru35}
		W. Krull, \textit{Idealtheorie}, Ergebnisse d. Math. Springer-Verlag, Berlin, 1968. MRNUMBER = {229623}
		
		\bibitem{KP08}
		D. Kruml and J. Paseka, Algebraic and categorical aspects of quantales, in \textit{Handbook of algebra}, vol. 5, Elsevier, 
		2008. DOI = {10.1016/S1570-7954(07)05006-1}
		
		\bibitem{LW14}
		Q. Luo and G. Wang, Roughness and fuzziness in quantales,
		\textit{Information sciences},  271 (2014), 14--30. DOI = {10.1016/j.ins.2014.02.105}
		
		\bibitem{Maj14}
		D. J. Majcherek, \textit{Multiplicative lattice versions of some results
			from Noetherian commutative rings}, PhD 
		dissertation, University of California Riverside, 2014. MRNUMBER = {3301714}
		
		\bibitem{Mul86}
		C. J. Mulvey, \&, \textit{Rend. Circ. Mat. Palermo (2) Suppl.} No. 12 (1986), 99--104. MRNUMBER = {853151}
		
		\bibitem{N18}
		P. Nasehpour, Some remarks on ideals of commutative semirings, \textit{Quasigroups and Related Systems}, 26 (2018), 281--298. MRNUMBER = {3892985}
		
		\bibitem{Ros90}  K. I. Rosenthal, \textit{Quantales and their applications}, Pitman research notes in
		Math. No. 234, Longman, Scientific 
		and  Technical, 1990. MRNUMBER = {1088258}
		
		\bibitem{Sto37}  M. H. Stone,   Applications of the theory of Boolean rings to general topology, \textit{Trans. 
			Am. Math. Soc.}, 41 (1937), 
		375--481. DOI = {10.2307/1989788}
		
		\bibitem{SJ22}
		S. Sarode and V. Joshi, $\mathcal{X}$-elements in multiplicative lattices--a generalization of $J$-ideals, $n$-ideals and $r$-ideals 
		in rings, \textit{Int. Electron. J. Algebra}, 32 (2022), 46--61. DOI = {10.24330/ieja.1102289}
		
		\bibitem{War37}
		M. Ward, Residuation in structures over which a multiplication is defined, \textit{Duke Math.
			J.}, 3 (1937), 
		627--636. DOI = {10.1215/S0012-7094-37-00351-X}
		
		\bibitem{WD39}  \bysame and R. P. Dilworth, Residuated lattices, \textit{Trans. Amer. Math. Soc.}, 45 (1939), 335--354. DOI = {10.2307/1990008}
		
		\bibitem{YX13}
		L. Yang and L. Xu, Roughness in quantales, \textit{Information sciences}, 220 (2013), 568--579. DOI = {10.1016/j.ins.2012.07.042}
	\end{thebibliography}
\end{document}